\documentclass[a4paper]{amsart}
\usepackage{amsmath,amssymb,csquotes,hyperref}
\usepackage{mathabx,mathrsfs}

\usepackage[matrix,arrow,curve,cmtip]{xy}
\usepackage{xcolor}

\let\: \colon
\def\Kb{{K^{\mathrm{sep}}}}
\def\Kbp{{K_\gp^{\mathrm{sep}}}}
\def\localOp{{{\mathcal O}_\gp}}
\newcommand{\tO}[1]{{#1}\cap\intO}
\def\Z{{\mathbb Z}}
\def\A{{\mathbb A}}
\def\C{{\mathbb C}}
\def\P{{\mathbb P}}

\def\Q{{\mathbb Q}}
\def\F{{\mathbb F}}
\def\Gm{{\mathbb G}_{\mathrm{m}}}

\def\gp{{\mathfrak p}}
\def\gq{{\mathfrak q}}
\def\ga{{\mathfrak a}}
\def\gb{{\mathfrak b}}
\def\gc{{\mathfrak c}}
\def\gf{{\mathfrak f}}
\def\gd{{\mathfrak d}}
\def\gf{{\mathfrak f}}
\def\ggg{{\mathfrak g}}
\def\gr{{\mathfrak r}}

\def\gm{{\mathfrak m}}
\def\LOp{{\Lambda_{\gpO_\gp,\gp}}}
\def\DR{{\mathrm{DR}}}
\def\Map{{\mathrm{Map}}}
\def\alg{{\mbox{-}\mathrm{alg}}}

\def\Hom{{\mathrm{Hom}}}
\def\End{{\mathrm{End}}}
\def\Aut{{\mathrm{Aut}}}
\def\lcm{{\mathrm{lcm}}}
\def\fin{{\mathrm{fin}}}
\def\unr{{\mathrm{unr}}}
\def\ord{{\mathrm{ord}}}
\def\Cl{{\mathrm{Cl}}}
\def\G{{G_K}}
\def\I{{I_K}}
\def\Gb{{\bar{G}}}
\def\Gp{{G_\gp}}
\def\Ip{{I_\gp}}

\def\Id{\mathrm{Id}}
\def\E{{E}}
\def\L{{L}}
\def\Pr{{P}}
\def\V{{U}}
\def\iff{{\;\Leftrightarrow\;}}
\def\id{{\mathrm{id}}}

\newcommand{\maxO}{\mathit{O}} 

\newcommand{\sheafO}{\mathcal{O}} 

\newcommand{\scalA}{{A}} 
\newcommand{\ringb}{{B}} 
\newcommand{\ringd}{{D}} 
\newcommand{\ringr}{{R}}

\newcommand{\gpO}{\scalA}
\def\intO{{\scalA}^{\mathrm{int}}}

\newcommand{\MCM}{\mathcal{M}_{\mathrm{CM}}}
\newcommand{\sE}{\mathcal{E}}
\newcommand{\sL}{\mathcal{L}}
\newcommand{\bP}{\mathbb{P}}

\newcommand{\fl}{\mathrm{fl}}
\newcommand{\per}[2]{{#1}(#2)}
\newcommand{\perfl}[2]{{#1}{(#2)}_{\fl}}
\newcommand{\perred}[2]{{#1}{(#2)}_{\mathrm{red}}}
\newcommand{\peru}[2]{{#1}^{(#2)}}
\newcommand{\perv}[2]{R^{[#1]}_{#2}}

\newcommand{\red}{\mathrm{red}}
\newcommand{\rcl}[2]{R_{#1}(#2)}
\newcommand{\cyc}{Z(\Pr,\gr)}
\newcommand{\rcm}{ray class\ }
\newcommand{\mx}{M}
\newcommand{\GL}{\mathrm{GL}}

\newcommand{\labeleq}[1]{{\,\buildrel #1\over=\,}}
\newcommand{\longlabelmap}[1]{{\,\buildrel #1\over\longrightarrow\,}}

\def\longrightisomap{\,{\buildrel \sim\over\longrightarrow}\,}
\newcommand{\longmap}{{\,\longrightarrow\,}}
\newcommand{\smcoprod}{{\,\scriptstyle\amalg\,}}
\newcommand{\spsmcoprod}{{\;\scriptstyle\amalg\;}}
\DeclareMathOperator{\Spec}{Spec}
\DeclareMathOperator{\colim}{\mathrm{colim}}

\newtheorem{thm}[subsection]{Theorem}
\newtheorem{proposition}[subsection]{Proposition}
\newtheorem{corollary}[subsection]{Corollary}
\newtheorem{lemma}[subsection]{Lemma}

\begin{document}
\title{Explicit class field theory and the algebraic geometry of $\Lambda$-rings}
\author[J.~Borger, B.~de~Smit]{James Borger, Bart de Smit}
\date{\today}
\email{james.borger@anu.edu.au, desmit@math.leidenuniv.nl}
\thanks{The first part of this paper is a revision of our 
note~\cite{Borger-de-Smit:Deligne-Ribet2-preprint}, which
has had some circulation since 2006 and has been used by Yalkinoglu~\cite{Yalkinoglu:DR-monoid} and
Uramoto~\cite{Uramoto:semi-galois-II}.}

\begin{abstract}
We consider generalized $\Lambda$-structures on algebras and schemes over the ring of integers $\maxO_K$ of a
number field $K$. When $K=\Q$, these agree with the $\lambda$-ring structures of algebraic K-theory. We then
study reduced finite flat $\Lambda$-rings over $\maxO_K$ and show that the maximal ones are classified in a
Galois theoretic manner by the ray class monoid of Deligne and Ribet. Second, we show that the periodic loci on
any $\Lambda$-scheme of finite type over $\maxO_K$ generate a canonical family of abelian extensions of $K$.
This raises the possibility that $\Lambda$-schemes could provide a framework for explicit class field theory,
and we show that the classical explicit class field theories for the rational numbers and imaginary quadratic
fields can be set naturally in this framework. This approach has the further merit of allowing for some precise
questions in the spirit of Hilbert's 12th Problem.

In an interlude which might be of independent interest,
we define rings of periodic big Witt vectors and relate 
them to the global class field theoretical mathematics of the rest of the paper.
\end{abstract}

\maketitle

\setcounter{tocdepth}{1}
\tableofcontents

\section{Introduction} Let $\scalA$ be a Dedekind domain with fraction
field $K$. Let $\Pr$ be a set of maximal ideals of $\scalA$ such that for each $\gp\in\Pr$,
the residue field $k(\gp)=\scalA/\gp$ has finite cardinality $N(\gp)$.
We will be most interested in the case where $K$ is a number field or local field, $\scalA$
is the full ring of integers $\maxO_K$ of $K$, and
$\Pr$ is the set $\mx_K$ of all maximal ideals of $\maxO_K$.

Let $\ringb$ be a (commutative) $\scalA$-algebra. Then for
each $\gp\in\Pr$ the algebra $\ringb/\gp \ringb= \ringb \otimes_\scalA k(\gp)$ over
$k(\gp)$ has a natural $k(\gp)$-algebra endomorphism
$F_\gp\colon x\mapsto x^{N(\gp)}$, which is called the Frobenius
endomorphism.  By a Frobenius lift on $\ringb$ at $\gp$ we mean an
$\scalA$-algebra endomorphism $\psi_\gp\:\ringb\to\ringb$ such that $\psi_\gp\otimes
k(\gp)=F_\gp$.  

We define a \emph{$\Lambda_{\scalA,\Pr}$-structure} on $\ringb$ to be a set map
$\Pr\to\End_{\scalA\alg}(\E)$, denoted $\gp\mapsto\psi_\gp$, such that 
\begin{enumerate}
	\item $\psi_\gp$ is a Frobenius lift at $\gp$ for each $\gp\in \Pr$.
	\item $\psi_\gp \circ\psi_\gq =\psi_\gq \circ\psi_\gp$ for all $\gp,\gq\in\Pr$.
\end{enumerate}
By a \emph{$\Lambda_{\scalA,\Pr}$-ring} we mean an $\scalA$-algebra with $\Lambda_{\scalA,\Pr}$-structure.
(In fact, this definition of $\Lambda_{\scalA,\Pr}$-structure 
is well-behaved only when $\ringb$ is torsion free as an $\scalA$-module,
but since all the $\Lambda_{\scalA,\Pr}$-rings in
this paper have that property, we will use the simple definition given above.
For the general one, see \cite{Borger:BGWV-I}.)

For example, if $\scalA$ is the ring of integers of a number field $K$ and 
$\ringb$ is the ring of integers of a subfield $L$ of the strict Hilbert
class field of $K$, then $\ringb$ has a unique $\Lambda_{\scalA,\Pr}$-structure: 
$\psi_\gp$ is the Artin symbol of $\gp$ in the field
extension $K\subseteq L$.

Observe that if $\gp$ satisfies $\ringb\otimes_A k(\gp)=0$, then the lifting condition (1)
is vacuous.  In particular, if $\ringb$ is an algebra over $K$, then any
commuting collection of $K$-automorphisms of $\ringb$ indexed by the
maximal ideals of $\Pr$ is a $\Lambda_{\scalA,\Pr}$-structure on $\ringb$.
At a different extreme, if $\Pr$ consists of one maximal ideal, for example if $\scalA$ is a local ring,
then the commutation condition (2) is vacuous.

When $\scalA=\Z$ and $\Pr=\mx_\Q$, Wilkerson and Joyal have shown that a
$\Lambda_{\scalA,\Pr}$-structure on a ring without $\Z$-torsion is the same as a
$\lambda$-ring structure in the sense of algebraic K-theory \cite{Wilkerson}\cite{Joyal:Lambda}. 
For instance, for any abelian group $M$ we
have a natural $\Lambda_{\Z,\Pr}$-structure on the group ring $\Z[M]$ given by
$\psi_{(p)}(m)=m^p$ for $m\in M$ and prime $p$. 
In an earlier paper \cite{Borger-deSmit:Integral-models}, we showed that
a $\Lambda_{\Z,\Pr}$-ring that is reduced
and finite flat over $\Z$ is a sub-$\Lambda_{\Z,\Pr}$-ring of $\Z[C]^n$ for some
finite cyclic group $C$ and integer $n\geq 0$. 
The proof uses the
explicit description of ray class fields over $\Q$ as cyclotomic
fields. 

Over a general number field, class field theory is less explicit, and
the generalizations we present in the present paper are consequently
less explicit.  However, we can still give a very similar criterion
for a finite \'etale $K$-algebra $\E$ with $\Lambda_{\scalA,\Pr}$-structure to admit an integral 
$\Lambda_{\scalA,\Pr}$-model,
by which we mean a sub-$\Lambda_{\scalA,\Pr}$-ring $\ringb\subseteq \E$ which is finite flat as an $\scalA$-module
such that the induced map $K\otimes_{\scalA}\ringb \to \E$ is a bijection. See theorem
\ref{global-thm-intro} below.

Let $\Id_{\Pr}$ denote the set of non-zero ideals of $\scalA$ divisible only by the primes in $\Pr$.
It as a monoid under ideal multiplication, the free commutative
monoid on the set $\Pr$.
Let $\Kb$ be a separable closure of $K$, and let $\G$ denote the
Galois group of $\Kb$ over $K$. It is a profinite group. By a $\G$-set
$X$ we mean a finite discrete set with a continuous $\G$-action.
By Grothendieck's formulation of Galois theory, a finite
\'etale $K$-algebra $\E$ is determined by the $\G$-set $S$ consisting
of all $K$-algebra homomorphisms $\E\to \Kb$.
Giving a $\Lambda_{\scalA,\Pr}$-structure on $\E$ then translates to giving a monoid map
$\Id_{\Pr}\to\Map_\G(S,S)$.  By giving $\Id_{\Pr}$ the discrete topology, we
see that the category of $\Lambda_{\scalA,\Pr}$-rings whose underlying $\scalA$-algebra is
a finite \'etale $K$-algebra is anti-equivalent to the category of
finite discrete sets with a continuous action of the monoid
$\G\times \Id_{\Pr}$.

Let us first consider the case where $\scalA$ is a complete discrete valuation ring and $\Pr$
consists of the single maximal ideal $\gp$. Then $\Id_{\Pr}$ is isomorphic as a monoid to the
monoid of non-negative integers under addition.  Let $\I\subseteq\G$ be
the inertia subgroup.  Then $\I$ is normal in $\G$ and $\G/\I$ is the
absolute Galois group of $k(\gp)$, which contains the Frobenius
element $F\in \G/\I$ given by $x\mapsto x^{N(\gp)}$. Thus, $F$ acts
on any $\G$-set on which $\I$ acts trivially.

\begin{thm}
\label{local-thm}
Suppose that $\scalA$ is a complete discrete valuation ring
and that $\Pr$ consists of the single maximal ideal $\gp$.
Let $\E$ be a finite \'etale $K$-algebra with a $\Lambda_{\scalA,\Pr}$-structure,
and let $S$ be the set of $K$-algebra maps from $\E$ to $\Kb$.
Then $K$ has an integral $\Lambda_{\scalA,\Pr}$-model if and only if the
action of $\G\times \Id_{\Pr}$ on $S$ satisfies the following two conditions:
\begin{enumerate}
\item the group $\I$ acts trivially on $S_{\unr}=\bigcap_{n\geq 0}\gp^n S$;
\item the elements $\gp\in \Id_{\Pr}$ and $F\in \G/\I$ act in the same way on $S_{\unr}$.
\end{enumerate}
\end{thm}

\smallskip\noindent
See section~\ref{sec:local} for the proof.

Next, consider the global case, where $\scalA$ is the ring of integers in a number
field.

In order to express our result, let us first recall the definition of the \rcm monoid. A cycle of $K$ is a
formal product $\gf=\prod_\gp\gp^{n_\gp}$, where the product ranges over all primes of $K$, both finite and
infinite, all $n_\gp$ are non-negative integers, only finitely many of which are non-zero, and we have
$n_\gp\in\{0,1\}$ for real primes $\gp$, and $n_\gp=0$ for complex primes $\gp$. The finite part of $\gf$ is
$\gf_{\fin}=\prod_{\gp<\infty}\gp^{n_\gp}$, which can be viewed as an ideal of $\scalA$. We write
$\ord_\gp(\gf)=n_\gp$.

For a cycle $\gf$, we say that two integral ideals $\ga,\gb$
are \emph{$\gf$-equivalent}, and write $\ga\sim_{\gf}\gb$, if the following two conditions are satisfied:
	\begin{enumerate}
		\item $\ga$ and $\gb$ have the same greatest common divisor $\gd$ with $\gf_{\fin}$
		\item $\ga\gd^{-1}$ and $\gb\gd^{-1}$ represent the same class in 
			the ray class group $\Cl(\gf\gd^{-1})$ of conductor $\gf\gd^{-1}$.
	\end{enumerate}
This is an equivalence relation, and we write $\DR_{\Pr}(\gf)$ for the quotient $\Id_{\Pr}/{\sim_\gf}$. Because
$\sim_{\gf}$ is preserved by multiplication of ideals, $\DR_{\Pr}(\gf)$ inherits a unique monoid structure from
$\Id_\Pr$. We call it the \emph{\rcm monoid (or Deligne--Ribet monoid)} of conductor $\gf$ supported at $\Pr$.
It was introduced in Deligne--Ribet \cite{Deligne-Ribet} in the case where $\Pr=\mx_K$ and every real place
divides $\gf$. For alternative definitions of $\gf$-equivalence and $\DR_{\Pr}(\gf)$, see section~\ref{sec:DR}.

Let us say that $\Pr$ is \emph{Chebotarev dense} if any element of any ray class group $\Cl(\gf)$ can be
represented by an ideal supported at $\Pr$, or equivalently by infinitely many such ideals. For example, by
Chebotarev's theorem, any set $\Pr$ consisting of all but finitely many maximal ideals is Chebotarev dense. 
Whenever
$\Pr$ is Chebotarev dense, any element of $\Cl(\gf)$ can be written as the class of an ideal supported at $\Pr$,
and hence gives a well-defined element of $\DR_\Pr(\gf)$. This defines a map $\Cl(\gf)\to \DR_{\Pr}(\gf)$, which
is in fact injective. Composing with the Artin symbol defines a map
	\begin{equation} 
		\label{GalDRmap}
		\G\longmap \Cl(\gf)\longmap \DR_{\Pr}(\gf),
	\end{equation}
and hence a surjective map
	\begin{equation} 
		\label{DR-recmap}
		\G\times \Id_{\Pr} \longmap \DR_{\Pr}(\gf)
	\end{equation}
whose restriction to the first factor is the map (\ref{GalDRmap})
and whose restriction to the second factor is the canonical quotient map
$\Id_{\Pr}\to\DR_{\Pr}(\gf)$.

\begin{thm}
\label{global-thm-intro}
Suppose that $K$ is a number field and that $\Pr$ is Chebotarev dense.
Let $\E$ be a finite \'etale $K$-algebra with a $\Lambda_{\scalA,\Pr}$-structure,
and let $S$ be the set of $K$-algebra maps from $\E$ to $\Kb$.
Then $\E$ has an integral $\Lambda_{\scalA,\Pr}$-model if and only
if there is a cycle $\gf$ of $K$ such that the action of
$\G\times \Id_{\Pr}$ on $S$ factors (necessarily uniquely)
through the map $\G\times \Id_{\Pr} \longmap \DR_{\Pr}(\gf)$ above.
\end{thm}

It follows that the category of such $\Lambda_{\scalA,\Pr}$-rings is anti-equivalent to
the category of finite discrete sets with a continuous action by the
profinite monoid $\lim_\gf \DR_{\Pr}(\gf)$, where
the inverse limit is taken over all cycles $\gf$ with respect to the canonical surjective
maps $\DR(\gf)\to \DR(\gf')$ when $\gf'\mid \gf$.
When $K=\Q$, $\scalA=\Z$, and $\Pr=\mx_\Q$, this limit is the monoid $\hat{\Z}^{\circ}$ of all profinite 
integers under multiplication. In this case, the theorem above reduces to the first theorem of
our earlier paper \cite{Borger-deSmit:Integral-models}.
It was Lenstra who suggested that the \rcm monoid could play this 
role over general number fields.

When $\E$ admits an integral $\Lambda_{\scalA,\Pr}$-model, there must be a maximal one. (See
section~\ref{sec:max-model}.) In the example above, $\Z[x]/(x^n-1)$ is the maximal integral model of
$\Q[x]/(x^n-1)$. This is the second theorem in~\cite{Borger-deSmit:Integral-models}. In
general, let us write $\rcl{\scalA,\Pr}{\gf}$ for the maximal integral $\Lambda_{\scalA,\Pr}$-model associated
to the free $\DR_{\Pr}(\gf)$-set on one generator, namely $\DR_{\Pr}(\gf)$. We call $\rcl{\scalA,\Pr}{\gf}$ the
\emph{ray class algebra} of conductor $\gf$---just as $K(\gf)$, the ray class field of conductor $\gf$, is the
extension of $K$ corresponding to the free $\Cl(\gf)$-set on one generator. The ray class algebra is an order in 
a product of ray class fields:
	$$
	K\otimes_{\scalA}\rcl{\scalA,\Pr}{\gf} = \prod_{\gd\mid\gf,\gd\in\Id_\Pr}K(\gf\gd^{-1}).
	$$
It is typically smaller than the maximal order in the non-$\Lambda$ sense.
For example,
	$$
	\rcl{\Z,\mx_\Q}{n\infty}=\Z[x]/(x^n-1)\subsetneq \prod_{d\mid n}\Z[\zeta_d].
	$$	

The theorems in~\cite{Borger-deSmit:Integral-models} for $K=\Q$, however, give us something more than the
abstract existence theorem above. They give explicit presentations of the ray class algebras
$\rcl{\Z,\mx_\Q}{n\infty}$, namely $\Z[x]/(x^n-1)$. More importantly, the presentations are all as quotients of
a single finitely generated $\Lambda$-ring---in this case $\Z[x^{\pm 1}]$, or $\Z[x]$, where each $\psi_p$ is
defined by $\psi_p(x)=x^p$. One can view this as a $\Lambda$-refinement of the Kronecker--Weber theorem, telling
us that the function algebra $\sheafO(\mu_n)$ of the $n$-torsion subscheme $\mu_n\subset\Gm=\Spec\Z[x^{\pm 1}]$
is isomorphic as a $\Lambda_{\Z,\mx_\Q}$-ring to the ray class algebra $\rcl{\Z,\mx_\Q}{n\infty}$; this is
instead of the statement that its set of $\bar{\Q}$-points, $\mu_n(\bar{\Q})$, generates the ray class field
$\Q(n\infty)$. This refinement gives us Frobenius lifts at all primes, even those dividing the conductor, by
treating integral structures with more care. But the ray class algebras will have zero divisors
and be non-normal over the primes dividing the conductor, which some might consider a drawback. Then again,
they are normal in a $\Lambda$-ring theoretic sense, by definition.

It is natural to ask whether something like this holds for number fields $K$ larger than $\Q$.
Do the ray class algebras
$\rcl{\scalA,\Pr}{\gf}$ have a common origin in the algebraic geometry of $\Lambda_{\scalA,\Pr}$-rings? If so,
this would give a systematic way of generating ray class algebras and hence ray class fields. Or more modestly,
is it at least true that the known explicit class field theories admit a $\Lambda$-refinement as above? There is
also a converse question: when does a $\Lambda_{\scalA,\Pr}$-structure on an $\scalA$-scheme $X$ give rise to
a family of abelian extensions, as the $\Lambda_{\Z,\mx_\Q}$-structure on $\Gm$ does?

The converse is the easier direction, and we will consider it first. Let $X$ be a (flat and separated)
$A$-scheme with a $\Lambda_{\scalA,\Pr}$-structure, by which we mean a commuting family of Frobenius lifts
$\psi_\gp\:X\to X$, for $\gp\in\Pr$. 
If we are going to follow the model of $\mu_n\subseteq\Gm$ and produce abelian extensions of our number field 
$K$ by finding finite flat sub-$\Lambda_{\scalA,\Pr}$-schemes $Z\subseteq X$ and applying the
theorem above, then by this theorem, there must exist a cycle $\gf$ such that the Frobenius operators $\psi_\ga$
on $Z$ are $\gf$-periodic in $\ga$, meaning
that they depend only on the class of $\ga$ in $\DR_\Pr(\gf)$. So it is natural to consider the largest such
subscheme, the locus $\per{X}{\gf}\subseteq X$ where the operations $\psi_\ga$ depend only the class
of $\ga$ in $\DR_\Pr(\gf)$. We call $\per{X}{\gf}$ the $\gf$-periodic locus. It is defined by a large equalizer
diagram; so it does indeed exist and is a closed subscheme of $X$.

For example, in the cyclotomic setting with the $\Lambda_{\Z,\mx_\Q}$-scheme $X=\Gm$ as above and $\gf=n\infty$,
with $n\geq 1$, the $\gf$-periodicity condition is $\psi_{(m+n)}=\psi_{(m)}$ for all $m\geq 1$. In other words,
it is that the operators $\psi_{(m)}$ are periodic in $m$ with period dividing $n$. It follows that the
$\gf$-periodic locus is just the $n$-torsion locus $\mu_n$. In general, while the $\gf$-periodic locus is
similar in spirit to the $\gf$-torsion locus when $X$ is a group, they can be different: for example if $X$ is
still $\Gm$ but $\gf$ is trivial at $\infty$, so $\gf=(n)$, then the periodic locus must also be invariant under
the involution $x\mapsto x^{\pm 1}$ and is hence just $\mu_2$ if $n$ is even, or $\mu_1$ if $n$ is odd. 

But the definition of $\per{X}{\gf}$ does not even require $X$ to be a group. Thus the group scheme structure in
the traditional frameworks for explicit class field theory is replaced by a $\Lambda$-structure in ours. This
will allow us some flexibility that is not available when working with group schemes. For example, we can divide
out a CM elliptic curve, say, by its automorphism group. Although the group structure is lost, the
$\Lambda$-structure is retained and hence we can still speak of the periodic locus on the quotient. Note that
whereas in the group-scheme setting, the abelian nature of the Galois theory comes from the torsion locus
being of rank $1$ over some commutative ring of complex multiplications and then from the commutativity of 
the general linear group $\GL_1$,
in our setting, it comes from the assumption that the Frobenius lifts $\psi_\gp$ commute with each
other and then from Chebotarev's theorem.

We can now state the answer to the converse question:

\begin{thm}
\label{thm:intro-periodic}
Let $X$ be a $\Lambda_{\scalA,\Pr}$-scheme of finite type over $\scalA$, as above, and assume $\Pr$ is Chebotarev dense.
Then possibly after inverting some primes, $\per{X}{\gf}$ is an affine $\Lambda_{\scalA,\Pr}$-scheme which
is reduced and finite flat over $A$. 
\end{thm}
The idea of the proof is that for primes $\gp\nmid\gf$, the endomorphism $\psi_\gp$ is an automorphism of finite
order because the class $[\gp]\in\DR_\Pr(\gf)$ is invertible and hence has finite order; therefore the Frobenius
endomorphism of the fiber of $X$ over $\gp$ is an automorphism of finite order, and so the fiber is finite and
geometrically reduced. Then apply the semicontinuity theorems of scheme theory. For details, see
theorem~\ref{thm:periodic-locus3}.

It then follows from theorem~\ref{global-thm-intro} that $\sheafO(\per{X}{\gf}_K)$,
the function algebra of the generic fiber of $X(\gf)$,
is a finite product of abelian extensions of $K$ of conductor dividing $\gf$. Thus any
$\Lambda_{\scalA,\Pr}$-scheme $X$ of finite type (still flat and separated) gives rise to a uniform geometric
way of constructing abelian extensions indexed by cycles $\gf$. It will not always produce arbitrarily large
abelian extensions---for example, the Chebyshev line below will only produce the ray class fields over $\Q$ with
trivial conductor at infinity, namely $\Q(\zeta_n+\zeta_n^{-1})$.

We can now state some precise versions of our original question of whether the ray class algebras have a common
origin in $\Lambda$-algebraic geometry. It is cleaner to restrict to cycles $\gf$ of a fixed type away from $\Pr$, and in particular of a fixed type at infinity. So fix a cycle $\gr$ supported away from $\Pr$, and let $\cyc$ denote the set of cycles $\gf$ which agree with $\gr$ away from the primes in $\Pr$.
We will refer to the triple $(\scalA,\Pr,\gr)$ as the \emph{context}.
\begin{enumerate}
	\item[(Q1)] Does there exist a $\Lambda_{\scalA,\Pr}$-scheme $X$ of finite type such that for all 
	$\gf\in\cyc$, the direct factors of the \'etale $K$-algebra $\sheafO(\per{X}{\gf}_K)$ generate the ray class 
	field $K(\gf)$?
\end{enumerate}

The Kronecker--Weber theorem states that the answer is positive in the cyclotomic context $(\Z,\mx_\Q,\infty)$,
with $X$ being $\Gm$ with the usual $\Lambda_{\Z,\mx_\Q}$-structure defined above. We will show it is true in
two other classical contexts of explicit class field theory, namely the real-cyclotomic context $(\Z,\mx_\Q,(1))$
and the  context $(\maxO_K,\mx_K,(1))$ where $K$ is an imaginary quadratic field.

In fact, we will show a stronger $\Lambda$-refinement holds. The strongest question one might ask is the 
following:
\begin{enumerate}
	\item[(Q2)] Does there exist a $\Lambda_{\scalA,\Pr}$-scheme $X$ of finite type such that 
	for all $\gf\in\cyc$, the function algebra $\sheafO(\per{X}{\gf})$ of the periodic locus
	is isomorphic to the ray class algebra $\rcl{\scalA,\Pr}{\gf}$?
\end{enumerate}

As mentioned above, the answer is positive in the cyclotomic context with $X=\Gm$.
But it appears to be slightly too much to ask for in general. For instance, consider
the real-cyclotomic context $(\Z,\mx_\Q,(1))$, and let $X=\Gm/(x\mapsto x^{-1})\cong\A^1_{\Z}$. Then the 
$\Lambda_{\Z,\mx_\Q}$-structure on
$\Gm$ descends to one on $X$. The operations $\psi_{(p)}$ are given by Chebyshev polynomials 
$\psi_{(p)}(y)\in\Z[y]$
determined by $\psi_{(p)}(x+x^{-1})=x^p+x^{-p}$. Then for any $n\geq 1$, the periodic locus $X(n)$ does 
indeed generate the real ray class fields,
	$$
	\Q(X(n)(\bar{\Q}))=\Q(\zeta_n+\zeta_n^{-1}),
	$$
but $\sheafO(X(n))$ does not generally agree with the full ray class algebra $\rcl{\Z,\Pr}{n}$. It is however
only of index $2$ in it, once $n$ is even, and this error even disappears in the limit as $n$ grows. So in the
real-cyclotomic context, the answer to (Q2) is as about as close to being positive without being so (at least
for the given $X$!). This discrepancy is no doubt due to the nontrivial isotropy groups of the points $\pm 
1\in\Gm$ under the involution $x\mapsto x^{-1}$, and it may very well disappear in a proper stack-theoretic 
treatment.

But if we stay in the world of schemes, we need to control it.
So given a finite reduced $\Lambda_{\scalA,\Pr}$-ring $B$, let $\tilde{B}$ denote the maximal
integral $\Lambda_{\scalA,\Pr}$-model in $K\otimes_A B$. It is the $\Lambda$-analogue of 
the integral closure of $B$.

\begin{enumerate}
	\item[(Q3)] Does there exist a $\Lambda_{\scalA,\Pr}$-scheme $X$ of finite type such that
	for all $\gf\in\cyc$, 
	\begin{enumerate}
		\item[(i)] the function algebra $\sheafO(\per{X}{\gf})$ is of finite index in
			$\sheafO(\per{X}{\gf})^{\sim}$, 
		\item[(ii)] the morphism of pro-rings
			$$
			\Big(\sheafO(\per{X}{\gf})\Big)_{\gf\in\cyc} \longmap 
				\Big(\sheafO(\per{X}{\gf})^{\sim}\Big)_{\gf\in\cyc}
			$$
			is an isomorphism, and
		\item[(iii)] $\sheafO(\per{X}{\gf})^{\sim}$ is isomorphic to $\rcl{\scalA}{\gf}$ as a 
			$\Lambda_{\scalA,\Pr}$-ring?
	\end{enumerate}

\end{enumerate}

\begin{thm}\label{thm:main}
	The answer to (Q3) is positive in the following contexts $(\scalA,\Pr,\gr)$ with the 
	$\Lambda_{\scalA,\Pr}$-schemes $X$:
	\begin{enumerate}
		\item $(\Z,\mx_\Q,\infty)$ and $X=\mathbb{G}_{\mathrm{m}/\Z}$ 
			with the $\Lambda_{\Z,\mx_\Q}$-structure above
		\item $(\Z,\mx_\Q,(1))$ and $X=\A^1_\Z$ with the Chebyshev $\Lambda_{\Z,\mx_\Q}$-structure
		\item $(\maxO_K,\mx_K,(1))$, where $K$ is imaginary quadratic with Hilbert class field $H$,
		and $X$ is $\P^1_{\maxO_H}$, viewed as a scheme over $\maxO_K$, with the Latt\`es 
		$\Lambda_{\maxO_K,\mx_K}$-structure defined in (\ref{subsec:lattes}).
	\end{enumerate}
\end{thm}

The $\Lambda$-scheme $\P^1_{\maxO_H}$ in the imaginary quadratic context plays the role of the target of the
Weber functions $E\to E/\Aut(E)\cong\P^1_\C$ in the traditional treatments of explicit class field theory of
imaginary quadratic fields. There are, however, a number of subtleties in constructing this $\Lambda$-structure.
For instance, CM elliptic curves are only defined over $H$ and not $K$, there can be more than one of them, but
it might be that none of them has good reduction everywhere. These problems were completely clarified in
Gurney's thesis~\cite{Gurney:thesis}. He also gave an account of class field theory for imaginary quadratic
fields from the point of view of the $\Lambda$-structure on $\P^1_{\maxO_H}$, but he stopped at considering the
field extensions generated by the preimages $\psi_\gf^{-1}(\infty)$. In either approach, the package of elliptic
curves with complex multiplication and their Weber functions is replaced by the single
$\Lambda_{\maxO_K,\mx_K}$-scheme $\P^1_{\maxO_H}$. We emphasize that $\P^1_{\maxO_H}$ is not so interesting from
a cohomological or motivic point of view---there are no interesting Galois representations in the cohomology.
All the richness is in the $\Lambda$-structure.

It would be interesting to know if something similar holds for other number fields, starting with CM fields. One
virtue of the questions above is that they allow for negative answers, which
would also be interesting. In the final section, we will raise some further questions in this direction.

\vspace{2mm}
It is a pleasure to thank Lance Gurney and Hendrik Lenstra for several helpful discussions.

\part{Finite $\Lambda$-rings and class field theory}

\section{Maximal $\Lambda$-orders}
\label{sec:max-model}

\subsection{} \emph{Maximal $\Lambda$-orders.}
Let $\scalA'$ be a sub-$\scalA$-algebra of $K$, and
let $E$ be a finite $K$-algebra with $\Lambda_{\scalA,\Pr}$-structure. A sub-$\scalA'$-algebra $B\subseteq E$
with a $\Lambda_{\scalA,\Pr}$-structure is said to be a \emph{$\Lambda_{\scalA,\Pr}$-order (over $\scalA'$)}
if it is finite over $\scalA'$. We say it is \emph{maximal (in $E$)} if it contains all others.
Maximal $\Lambda_{\scalA,\Pr}$-orders always exist, by an elementary argument 
(\cite{Borger-deSmit:Integral-models}, prop.\ 1.1), since maximal orders in the ordinary sense exist
and since $\scalA'$ is noetherian.

We will say a finite flat $\scalA'$-algebra $B$ with a $\Lambda_{\scalA,\Pr}$-structure is \emph{normal} (or
\emph{$\Lambda_{\scalA,\Pr}$-normal over $\scalA'$}) if it is maximal in $K\otimes_{\scalA'} B$, in the sense 
above.

We will need the following basic facts:

\begin{proposition}\label{pro:maximality-facts}
	Let $B$ be a finite flat $\scalA'$-algebra with $\Lambda_{\scalA,\Pr}$-structure.
	\begin{enumerate}
		\item $B$ is $\Lambda_{\scalA,\Pr}$-normal over $\scalA'$ if 
		$\maxO_{K_\gp}\otimes_{\scalA'}B$ is $\Lambda_{\maxO_{K_\gp},\gp}$-normal over $\maxO_{K_\gp}$ 
		for all	maximal ideals $\gp$ of $\scalA'$. 
		\item If $B$ is $\Lambda_{\scalA,\Pr}$-normal over $\scalA'$, and $G$ is a group acting on $B$ by 
		$\Lambda_{\scalA,\Pr}$-automorphisms, then the invariant subring $B^G$ is $\Lambda_{\scalA,\Pr}$-normal 
		over $\scalA'$.
		\item If $B$ is a product $B_1\times B_2$ of $\Lambda_{\scalA,\Pr}$-rings, then $B$
		is $\Lambda_{\scalA,\Pr}$-normal over $\scalA'$ if and only if $B_1$ and $B_2$ are.
	\end{enumerate}
\end{proposition}
\begin{proof}

(1): Let $C\subseteq K\otimes_{\scalA'}B$ be a
$\Lambda_{\scalA,\Pr}$-order over $\scalA'$
containing $B$. Then for any maximal ideal $\gp\subset \scalA'$,
the base change $\maxO_{K_\gp}\otimes_{\scalA'}C$
is a $\Lambda_{\maxO_{K_\gp},\gp}$-order over $\maxO_{K_\gp}$. Therefore
it agrees with $\maxO_{K_\gp}\otimes_{\scalA'}B$. Since this holds for all maximal ideals $\gp\subset\scalA'$,
the ring $C$ agrees with $B$. Therefore $B$ is maximal.

(2): Let $C\subseteq K\otimes_{\scalA'}B^G$ be a $\Lambda_{\scalA,\Pr}$-order over $\scalA'$ containing $B^G$. Then $C$ is contained in $K\otimes_{\scalA'}B$ and hence, by the maximality of $B$, is contained in $B$. But since $C$ is contained in $K\otimes_{\scalA'}B^G$,
it is also $G$-invariant. Therefore we have $C\subseteq B^G$.

(3): Let $C\subseteq K\otimes_{\scalA'}(B_1\times B_2)$ be a
$\Lambda_{\scalA,\Pr}$-order over $\scalA'$
containing $B_1\times B_2$. Put $C_i=C\otimes_{B_1\times B_2} B_i$, for $i=1,2$.
Then each $C_i$ is a $\Lambda_{\scalA,\Pr}$-order over $\scalA'$ in $K\otimes_{\scalA'}B_i$.
By the maximality of $B_i$, we have $C_i=B_i$. Therefore we have
$$
C = C_1\times C_2 = B_1\times B_2.
$$
Thus $B_1\times B_2$ is maximal.
\end{proof}

\section{The local case}
\label{sec:local}

In this section, $\scalA$ will be a complete discrete valuation ring with maximal
ideal $\gp$, and $\Pr$ will be the singleton set
$\{\gp\}$. So $\Id_{\Pr}$ is the multiplicative monoid
of all nonzero ideals of $\scalA$.
Write $k=k(\gp)$ and $\Lambda_{\scalA,\gp}=\Lambda_{\scalA,\Pr}$.  

\begin{proposition}
\label{unramified}
Let $\ringb$ be a finite \'etale $\scalA$-algebra. 
Then $\ringb$ has a unique
$\Lambda_{\scalA,\gp}$-structure, and the induced action of $\G\times \Id_{\Pr}$ on
$S=\Hom_{K\alg}(\ringb\otimes_\scalA K, \Kb)$ has the
property that the inertia group $\I$ acts trivially and that the element
$\gp\in\Id_{\Pr}$ and the Frobenius element $F\in\G/\I$ act on $S$ in the same way.
\end{proposition}

\begin{proof}
Because $\ringb$ is \'etale, $k\otimes_\scalA \ringb$ is a product
of finite fields.  Since $\ringb$ is complete in its $\gp$-adic topology,
idempotents of $\ringb/\gp \ringb$ lift to $\ringb$, so that $\ringb$ is a finite product
of rings of integers in finite unramified extensions of $K$. 
Therefore
the inertia group $\I\subseteq \G$ acts trivially on  $S=\Hom_{\scalA\alg}(\ringb,\Kb)$.  Every
finite unramified field extension $\L$ of $K$ is Galois with an abelian
Galois group, and its rings of integers has a unique Frobenius lift.
It follows that when $\ringb$ is unramified over $\scalA$, it has a
unique $\Lambda_{\scalA,\gp}$-structure. 
\end{proof}

\subsection{} \label{localstructure}
\emph{Structure of finite \'etale $K$-algebras with $\Lambda_{\scalA,\gp}$-structure.}
Let $\E$ be a finite \'etale $K$-algebra with a $\Lambda_{\scalA,\gp}$-structure,
and write 
	$$
	S=\Hom_{K\alg}(\E,\Kb).
	$$
Put $S_0=\bigcap_{n\geq 0}\gp^n S$ and for $i=1,2, \ldots$, let
	$$
	S_i = \{s\in S\colon\; s\not\in S_{i-1} \text{ and } \gp s\in S_{i-1}\}.
	$$
Then each $S_i$ is a sub-$\G$-set of $S$.
Since $S$ is finite, there exists an $n\geq 0$ such that  $S_{n+1}=\varnothing$.
Then we have the decomposition
	$$
	S= S_0\smcoprod S_1 \smcoprod \cdots \smcoprod S_n.
	$$
Let $\E_i=\Map_{\G}(S_i, \Kb)$ be the corresponding finite \'etale
$K$-algebra for each $i$. Then the maps $S_n\to \cdots \to S_1\to S_0\righttoleftarrow$
given by multiplication by $\gp$
give rise to a diagram of $K$-algebras
\[
	f_0 \lefttorightarrow
	\E_0 {\buildrel f_1 \over\longrightarrow}
	\E_1 {\buildrel f_2 \over\longrightarrow}
	\cdots
	{\buildrel f_{n} \over\longrightarrow}
	\E_n
\]
Since $S=S_0\smcoprod S_1 \smcoprod \cdots\smcoprod S_n$ is a
decomposition of $G$ as a $\G$-set, we have a corresponding
product decomposition of the finite \'etale $K$-algebras
$\E=\E_0\times \E_1 \times \cdots\times \E_n$. 
In terms of this decomposition, $\psi_p$ is given by
	\begin{equation} \label{psicomp}
		\psi_\gp(e_0,e_1, \ldots, e_n)=(f_0(e_0),f_1(e_0),\ldots, f_{n}(e_{n-1})).
	\end{equation}
Since $S_0$ is closed under multiplication by $\gp$, the quotient ring
$\E_0$ of $\E$ is a quotient $\Lambda_{\scalA,\gp}$-ring of $\E$, with Frobenius lift
$f_0$. 

We will now construct a splitting of this quotient map $\E\to \E_0$ in the category
of $\Lambda_{\scalA,\gp}$-rings.
Note that we have $\gp^kS=S_0$ for sufficiently large
$k$; so we have $\gp S_0=S_0$ and hence $\gp$ act as a bijection on $S_0$. Thus, $f_0$ is an
automorphism of $\E_0$.
For $s\in S_i$ we have $\gp^is\in S_0$. Again since $\gp$ acts bijectively on
$S_0$, we can define a map 
$S\longmap S_0$ by sending $s\in S_i$ to 
the unique element $s'$ of $S_0$ such that $\gp^i s' = \gp^i s$.
This map commutes with the $\G\times \Id_{\Pr}$-action, and it is a retraction of
the inclusion $S_0\to S$. This induces our desired splitting $\E_0\to E$.
In other words, $\E_0$ is not only a quotient
 $\Lambda_{\scalA,\gp}$-ring of $\E$, but also a sub-$\Lambda_{\scalA,\gp}$-ring:
\begin{align*}
	j\colon\; \E_0  \longmap  & \E \\
	e_0  \longmapsto & (e_0,f_1f_0^{-1}e_0, f_2f_1f_0^{-2} e_0, \ldots,
	f_{n-1} \cdots f_1 f_0^{-n+1} e_0).
\end{align*}

\subsection{} \emph{Proof of theorem \ref{local-thm}.}
Suppose that $\E$ has an integral $\Lambda_{\scalA,\gp}$-model $\ringb$, that is, 
\begin{enumerate}
	\item[(i)] $\ringb$ is finitely generated as an $\scalA$-module and has rank $\dim_{K}(\E)$,
	\item[(ii)] $\psi_\gp(\ringb)\subseteq \ringb$,
	\item[(iii)] $\psi_\gp\otimes_\scalA k$ is the Frobenius map $x\mapsto x^{N(\gp)}$
		on $\ringb\otimes_A k$.
\end{enumerate}

Let $\ringb_0$ denote the image of $\ringb$ under the quotient map $\E\to\E_0$
(in the notation of \ref{localstructure}). Then $\ringb_0$ is an
integral $\Lambda_{\scalA,\gp}$-model of $\E_0$. Since $f_0$ is an automorphism of $\E_0$, the ring $\ringb_0$
and its subring $f_0(\ringb_0)$ have the same discriminant. Thus,
$f_0(\ringb_0)=\ringb_0$ and hence $f_0$ is an automorphism of $\ringb_0$.  This implies
that the map $x\mapsto x^{N(\gp)}$ on $\ringb_0\otimes_\scalA k$ is an
automorphism, and so $\ringb_0$ is \'etale over $\scalA$.  Conditions (1)
and (2) of theorem \ref{local-thm} now follow by proposition
\ref{unramified}.

For the converse, suppose that conditions (1) and (2) hold.
We will produce an integral $\Lambda_{\scalA,\gp}$-model of $\E=\E_0\times \cdots\times \E_n$.
Let $R_i$ be the integral closure of $\scalA$ in $\E_i$.
Since $\I$ acts trivially on $S_\unr=S_0$, the $A$-algebra $R_0$ is finite \'etale and hence has a
unique $\Lambda_{\scalA,\gp}$-structure by proposition \ref{unramified}. 
Our integral model $\ringb\subseteq E$ will be of the form
\begin{equation}
	\label{eq:lambda-split}
\ringb = j(R_0) \oplus  (0\times \ga_1\times \cdots\times\ga_n),
\end{equation}
where each $\ga_i$ is an ideal in $R_i$ and $j$ is the map defined in~\ref{localstructure}. 
Observe that any $\ringb$ of this form is a subring of $E$.
For it to be a sub-$\Lambda_{\scalA,\gp}$-ring, we need to have $\psi_\gp(a)\equiv a^{N(\gp)}\bmod \gp \ringb$
for all $a\in \ringb$. Since both sides of this congruence are additive in $a$, it is enough to consider elements $a$ in each of the summands in~(\ref{eq:lambda-split}).
It holds for the summand $j(R_0)$ because $j$ is a $\Lambda_{\scalA,\gp}$-morphism.
So by (\ref{psicomp}), a sufficient condition for $\ringb$ to be a sub-$\Lambda_{\scalA,\gp}$-ring is
$f_i(\ga_{i-1})\subseteq \gp\ga_{i}$ and
$\ga_i^{N(\gp)}\subseteq\gp\ga_i$ for $i=1,\dots,n$, where we take $\ga_0=0$. This holds if, for instance,
$\ga_i=\gp^{n-i+1}R_i$. So for this choice, $\ringb$ is an integral $\Lambda_{\scalA,\gp}$-model of $\E$.
\qed

\subsection{} \emph{Remark.}
\label{rmk-group-alg}
Note that the integral model produced in the proof above is not always the
maximal one. For instance, if $C_n$ denotes a cyclic group of order $n$, then on
the group algebra $\Q_2[C_4]$ with its usual
$\Lambda_{\Z_2,2}$-structure, the proof produces a strict
subring of $\Z_2[C_4]$ and hence cannot be not maximal. (In fact, $\Z_2[C_4]$ is the maximal
integral model, as is shown in~\cite{Borger-deSmit:Integral-models}).

It can also happen that, for group algebras, the integral
model supplied by the proof is strictly larger than 
the integral group algebra. An example is $\Q_2[C_2\times C_2]$.

\subsection{} \emph{Remark.}
It is possible to express theorem \ref{local-thm} in a more Galois-theoretic way, similar
to the statement of theorem \ref{global-thm-intro}. We can define an inverse system
of finite quotients $G_{N,n}$ of the monoid $\G\times\Id_{\Pr}$ with
the property that $\E$ has an integral $\Lambda_{\scalA,\gp}$-model
if and only if the action of $\G\times \Id_{\Pr}$ on $S$ factors through some $G_{N,n}$.

The quotients $G_{N,n}$ are defined as follows.
Let $N$ be an open normal subgroup of $\G$, and let $n$ be an integer $\geq 0$.
Define a relation on $\G\times \Id_{\Pr}$ by $(g,\gp^a)\sim (h,\gp^b)$
if either or both of the following conditions hold:
	\begin{enumerate}
		\item $a=b$ and $g\equiv h \bmod N$
		\item $a,b\geq n$ and $g F^a\equiv h F^{b} \bmod N\I$.
	\end{enumerate}
This is easily seen to be an equivalence relation which is stable under the monoid
operation. We then define $G_{N,n}$ to be the quotient monoid. Observe that we have
a decomposition of $\G$-sets:
\begin{equation}
	\label{eq:G_N,n}
	G_{N,n} = \underbrace{\G/N \spsmcoprod \cdots \spsmcoprod \G/N}_{n\text{ times}} \spsmcoprod \G/N\I.
\end{equation}

For $N'\subseteq N$ and $n'\geq n$, we have evident transition
maps $G_{N',n'} \to G_{N,n}$. If we consider the inverse limit
	$$
	\hat{G}=\lim_{N,n} G_{N,n},
	$$
then the re-expression of theorem \ref{local-thm} is that $\E$ has an integral $\Lambda_{\scalA,\gp}$-model
if and only if the action of $\G\times\Id_{\Pr}$ on $S$ factors (necessarily uniquely) through a continuous
action of $\hat{G}$. One might call $\hat{G}$ the $\Lambda_{\scalA,\gp}$-algebraic fundamental monoid of 
$\Spec \maxO_K$ with ramification allowed along $\Spec k$.

\subsection{} \emph{Remark.} 
In the global case considered in the rest of this paper, we will see that only abelian field extensions arise
from integral $\Lambda$-models. But in the local case, we can get nonabelian extensions. In fact, we can get
arbitrary extensions. Indeed, any given extension $L$ of $K$ is a direct factor of the $K$-algebra $L_\unr\times
L$, which has a $\Lambda_{\scalA,\gp}$-structure admitting an integral model. For instance, one can take
$\psi_\gp(e_0,e_1)=(F(e_0),e_0)$.

\section{The \rcm monoid}
\label{sec:DR}
Fix the following notation:
	\begin{align*}
		K				&= \text{a finite extension of $\Q$} \\
		\maxO_K			&= \text{the ring of integers in $K$} \\
		\mx_K			&= \text{the set of maximal ideals of $\maxO_K$} \\
		\scalA			&= \text{a Dedekind domain whose fraction field is $K$} \\
		\Pr				&= \text{a set of maximal ideals of $\scalA$} \\
		\Id_{\Pr}		&= \text{the monoid of nonzero ideals of $\scalA$ supported at $\Pr$} \\
		\gf				&= \text{a cycle (or modulus) on $K$} \\
		\gf_{\Pr}		&= \text{the part of $\gf$ supported at $\Pr$} \\
		\gf^{\Pr}		&= \text{the part of $\gf$ supported away from $\Pr$} \\
		\gf_{\fin}		&= \text{the part of $\gf$ supported at the finite primes} \\
		\gf_{\infty}	&= \text{the part of $\gf$ supported at the real places} \\
		\Id_{\Pr}(\gf)	&= \text{the submonoid of $\Id_{\Pr}$ of ideals prime to $\gf_{\fin}$} \\
		\Cl(\gf)		&= \text{the ray class group of $K$ of conductor $\gf$} \\
		\Cl_{\Pr}(\gf)	&= \text{the image of the canonical map $\Id_{\Pr}(\gf)\to \Cl(\gf)$} \\
		R^\circ			&= R \text{ viewed as a monoid under multiplication, for any ring }R 
	\end{align*}
Observe that, up to canonical isomorphism, the constructions above depend only on the places
of $K$ corresponding to $\Pr$, and so they depend on $\scalA$ only in that these places must come from
maximal ideals of $\scalA$. Therefore we can take $\scalA=\maxO_K$ without changing
anything above. 

\subsection{} \emph{Structure of the \rcm monoid.}
There is a bijection 
	\begin{equation} 
	\coprod_{\gd\in\Id_{\Pr},\gd\mid\gf} \Cl_{\Pr}(\gf\gd^{-1}) \longrightisomap \DR_{\Pr}(\gf),
	\end{equation}
sending an ideal class $[\ga]\in \Cl_{\Pr}(\gf\gd^{-1})$ in the summand of index $\gd$ to the class
$[\gd\ga]\in\DR_{\Pr}(\gf)$.
Thus we have
	\begin{equation} 
	\label{DR-decomp}
	\coprod_{\gd\in\Id_{\Pr},\gd\mid\gf} [\gd]\Cl_{\Pr}(\gf\gd^{-1}) = \DR_{\Pr}(\gf).
	\end{equation}
The multiplication law on $\DR_{\Pr}(\gf)$ is given in terms of the left-hand side by the formula 
	\begin{equation}
		\label{eq:DR-mult}
	[\gd][\ga]\cdot [\gd'][\ga'] = [\gd''][\ga''],
	\end{equation}
where $\gd''=\gcd(\gd\gd',\gf)$ and $\ga''$ satisfies $\gd''\ga''=\gd\ga\gd'\ga'$.
It follows, for example, that the submonoid $\DR_{\Pr}(\gf)^*\subseteq \DR_{\Pr}(\gf)$ of invertible elements 
agrees with the part of the ray class group supported at $\Pr$:
	$$
	\DR_{\Pr}(\gf)^*=\Cl_{\Pr}(\gf).
	$$

When $\Pr$ is Chebotarev dense, we have $\Cl_{\Pr}(\gf)=\Cl(\gf)$, and so
the invertible part of $\DR_{\Pr}(\gf)$ is independent of $\Pr$. Observe however that this is not the case for 
the noninvertible part. For example,
if $K=\Q$ and $\gf=(n)\infty$, then we have
	\begin{equation}
	\label{eq:DR-over-Q}
		\DR_{\Pr}(\gf) = (\Z/n_{\Pr}\Z)^{\circ} \times
			(\Z/n^{\Pr}\Z)^*,
	\end{equation}
where $n_{\Pr}$ is the factor of $n$ supported at $\Pr$ and $n^\Pr$ is that supported
away from $\Pr$. The invertible part is $(\Z/n_{\Pr}\Z)^* \times (\Z/n^{\Pr}\Z)^* = (\Z/n\Z)^*$, which
does not depend on $\Pr$, but the non-invertible part does.

\subsection{} \emph{Change of conductor.}
\label{subsec:DR-change-of-f}
Suppose $\gf\mid\gf'$. Then we have the implication
	$$
	\gb\sim_{\gf'}\gc \Rightarrow \gb\sim_{\gf}\gc.
	$$
which induces a monoid map
	\begin{equation}
		\DR_\Pr(\gf')\to \DR_{\Pr}(\gf),
	\end{equation}
which we call the canonical map.

There is also a map in the other direction. Write $\gf'=\gf\ga$.
Then by the equivalence 
	$$
	\gb\sim_\gf\gc \Leftrightarrow \ga\gb\sim_{\ga\gf}\ga\gc,
	$$
there is a well-defined injective map
	\begin{equation}
		\DR_\Pr(\gf) \longmap \DR_\Pr(\gf'), \quad [\gb]\mapsto [\ga\gb].
	\end{equation}
It is not a monoid map, but it is an equivariant map of $\DR_{\Pr}(\gf')$-sets.
The composition with the canonical projection $\DR_\Pr(\gf')\to \DR_{\Pr}(\gf)$
on either the left or the right, is given by multiplication by the class of $\ga$. 

\vspace{5mm}

We conclude this section with two more descriptions of $\gf$-equivalence and the \rcm monoid, although they only
have small parts in this paper. The first is the one that appears in Deligne--Ribet~\cite{Deligne-Ribet}, and a
form of the second was pointed out to us by Bora Yalkinoglu.

\begin{proposition}
\label{pro:DR-original-def}
	Two ideals $\ga,\gb$ of $\maxO_K$ are $\gf$-equivalent if and only if $\ga=x\gb$
	for some element $x\in 1+\gf_{\fin}\gb^{-1}$ which is positive at all real places dividing $\gf$.
\end{proposition}
\begin{proof}
	Let $\gd=\gcd(\gf,\gb)$.
	By definition, we have $\ga\sim_\gf\gb$ if and only if there exists an element
	$x\in K^*$ satisfying the following:
	\begin{enumerate}
		\item $\ga\gd^{-1}=x\gb\gd^{-1}$
		\item $x\equiv 1\bmod \gp^{n_\gp}$, where $n_\gp=\ord_\gp(\gf\gd^{-1})$, whenever $n_\gp\geq 1$
		\item $x$ is positive at all real places dividing $\gf\gd^{-1}$.
	\end{enumerate}
	Observe that (1) is equivalent to $\ga=x\gb$ and that (3) is equivalent to the positivity condition in the
	statement of the proposition.
	It is therefore enough to show that, under (1),
	condition (2) is equivalent to $x\in 1+\gf_{\fin}\gb^{-1}$, or equivalently to the condition that
	that for all $\gp$, we have 
	$x\equiv 1\bmod \gp^{m_\gp}$, where $m_\gp=\ord_\gp(\gf\gb^{-1})$. So fix a prime $\gp$.
	In the case $\ord_\gp(\gf)\geq \ord_\gp(\gb)$, we have $n_\gp=m_\gp$ and so this condition is indeed
	equivalent to (2).
	
	Now consider the remaining case $\ord_\gp(\gf)<\ord_\gp(\gb)$. Then we have $n_\gp=0$ and $m_\gp<0$.
	Because $n_\gp=0$,
	condition (2) is vacuous. Therefore it is enough to show that $x\equiv 1\bmod \gp^{m_\gp}$ necessarily 
	holds.
	Since $\gcd(\gf,\gb)=\gcd(\gf,\ga)$, we have $\ord_\gp(\gf)<\ord_\gp(\ga)$ and hence
	$$\ord_\gp(x-1)\geq \min\{\ord_\gp(x),0\}= \min\{\ord_\gp(\ga\gb^{-1}),0\} \geq \min\{m_\gp,0\} = m_\gp.$$
\end{proof}

\begin{proposition}
	\label{pro:Bora-DR-definition}
	Assume that $\Pr$ is Chebotarev dense. 	Then there is an isomorphism
			\begin{equation}
			\label{map:bora}
				(\maxO_K/\gf_{\Pr})^{\circ}\oplus_{(\maxO_K/\gf_{\Pr})^*}
				\Cl(\gf)\longrightisomap\DR_{\Pr}(\gf),
			\end{equation}
	which is given by the canonical inclusion $\Cl(\gf)\to\DR_\Pr(\gf)$ on the second factor and which
	on the first factor sends the residue class of any element $x\in\maxO_K$ with $(x)\in\Id_\Pr$ to the
	class $[(x)]\in\DR_\Pr(\gf)$.
\end{proposition}

The notation $A\oplus_G B$ above refers to the push out in the category of commutative monoids,
which when $G$ is a group is the quotient of $A\oplus B=A\times B$ by the action of $G$ given by 
$g\cdot(a,b)=(ga,g^{-1}b)$.

\begin{proof}
First observe that this morphism is well defined. Indeed, any element $x\in (\maxO_K/\gf_\Pr)^*$ is the image
of an element $\tilde{x}\in\maxO_K$ relatively prime to $\gf_\Pr$, and its class $[(\tilde{x})]$ in $\DR_\Pr(\gf)$
is indeed the image of its class in the ray class group $\Cl(\gf)$.

The fact that this morphism is an isomorphism is a consequence of the following equalities, which
will be justified below, and where $\gd$ runs over $\Id_\Pr$ with $\gd\mid\gf$:
\begin{equation}
	\label{eq:sss}
	\DR_\Pr(\gf) = \coprod_{\gd} [\gd]\Cl_\Pr(\gf\gd^{-1}) = \coprod_{\gd} [\gd]\Cl(\gf\gd^{-1}) 
\end{equation}
Let us simplify the right-hand side further and show 
	\begin{equation}
		\label{eq:ttt}
	[\gd]\Cl(\gf\gd^{-1}) = \Cl(\gf)/(1+\gf_\fin\gd^{-1}/\gf_\fin).		
	\end{equation}
We use the ad\`elic description of ray class groups: $\Cl(\gm)=\A_K^*/K^*\V_\gm$. Then we have the exact sequence
	$$
	\V_{\gf\gd^{-1}}/\V_\gf \longmap \Cl(\gf) \longmap \Cl(\gf\gd^{-1}) \longmap 1.
	$$
Since the archimedean parts of $\gf\gd^{-1}$ and $\gf$ agree, we have
	$$
	\V_{\gf\gd^{-1}}/\V_\gf = \V_{\gf_\fin\gd^{-1}}/\V_{\gf_\fin} = 1+\gf_\fin\gd^{-1}/\gf_\fin.
	$$
Equation (\ref{eq:ttt}) follows.

Combining this with (\ref{eq:sss}), we have
\begin{align*}
	\DR_\Pr(\gf) &\labeleq{} \coprod_{\gd} [\gd]\Cl(\gf)/(1+\gf_\fin\gd^{-1}/\gf_\fin) \\
				&\labeleq{} \coprod_{\gd} [\gd](\maxO_K/\gf_\fin\gd^{-1})^*\oplus_{(\maxO_K/\gf_\fin)^*}\Cl(\gf) \\
				&\labeleq{} \Big(\coprod_{\gd} [\gd](\maxO_K/\gf_\fin\gd^{-1})^*\Big)\oplus_{(\maxO_K/\gf_\fin)^*}\Cl(\gf) \\
				&\labeleq{} \big((\maxO_K/\gf_{\Pr})^{\circ}\times (\maxO_K/\gf^{\Pr})^*\big) 
					\oplus_{(\maxO_K/\gf_\fin)^*}\Cl(\gf) \\
				&\labeleq{} \big((\maxO_K/\gf_{\Pr})^{\circ}\times (\maxO_K/\gf^{\Pr})^*\big) 
					\oplus_{(\maxO_K/\gf_\Pr)^*\times (\maxO_K/\gf^\Pr)^*}\Cl(\gf) \\
				&\labeleq{} (\maxO_K/\gf_{\Pr})^{\circ} \oplus_{(\maxO_K/\gf_\Pr)^*}\Cl(\gf). 
\end{align*}

\end{proof}

\subsection{} \emph{Examples.}
If $K$ has class number $1$ and $\Pr$ is still Chebotarev dense, then we have
	\begin{equation}
	\label{eq:DR-PID2}
		\DR_{\Pr}(\gf) = 
			(\maxO_K/\gf_{\Pr})^{\circ}\oplus_{\maxO_{K,\gf_\infty}^*} (\maxO_K/\gf_\fin^{\Pr})^*,
	\end{equation}
where $\maxO_{K,\gf_\infty}^*$ is the subgroup of $\maxO_K^*$ consisting of units which are positive at all
places dividing $\gf_\infty$. 
At a different extreme, if $K$ is arbitrary but $\Pr=\mx_K$, then in the limit we have
\begin{equation}
	\label{eq:Bora-DR}
	\lim_\gf \DR_\Pr(\gf) = \hat{\maxO}_K^\circ \oplus_{\hat{\maxO}_K^*}G_K^{\mathrm{ab}}.
\end{equation}
If $K$ has class number $1$ and $\Pr=\mx_K$, then we have 
$$
\DR_{\Pr}(\gf) = (\maxO_K/\gf_\fin)^{\circ}/\maxO_{K,\gf_\infty}^*.
$$

\section{Global arguments}
The purpose of this section is to prove theorem \ref{global-thm-intro} from the introduction.
It will follow immediately from proposition \ref{abelian} and theorem \ref{global-thm2} below.

We continue with the notation of the previous section.
Also fix the following notation:
	\begin{align*}
		\E			&= \text{a finite \'etale $K$-algebra with a $\Lambda_{\scalA,\Pr}$-structure} \\
		S			&= \Hom_K(\E,\Kb), \text{ with its continuous action of } \G\times \Id_{\Pr}.
	\end{align*}
Define $\gr\in \Id_{\Pr}$ by setting
\begin{equation}
	\label{equ:r-def}
	\ord_\gp(\gr)=\inf\{i\ge 0\colon\; \gp^{i+1}S=\gp^i S\}
\end{equation}
for each prime $\gp\in\Pr$. This is well defined because
$\gp S=S$ whenever $\gp$ is unramified in $\ringb$, by proposition \ref{unramified}.

\begin{lemma} \label{gcd}
	Assume that $\E$ has an integral $\Lambda_{\scalA,\Pr}$-model $\ringb$.
	Let $\gd$ be an ideal in $\Id_{\Pr}$, and put $\gb=\gcd(\gd,\gr)$. Then $\gd S$
	equals $\gb S$, and this $G_K$-set is unramified at all primes dividing $\gd\gb^{-1}$.
\end{lemma}
\begin{proof}
	Observe that for any prime $\gp\mid\gd\gb^{-1}$, we have $\gp\gb S=\gb S$.
	Indeed, we have $\ord_{\gp}(\gb)=\ord_{\gp}(\gr)$ and hence
	$\gp^{1+\ord_{\gp}(\gb)}S=\gp^{\ord_{\gp}(\gb)}S$, by the definition of $\gr$.
	Then $\gp\gb S=\gb S$ follows.	

	This implies by induction that $\gd S=\gb S$.
	It also implies that for each prime
	$\gp\mid\gd\gb^{-1}$, the action
	of $\gp$ on $\gb S=\gd S$ is bijective.
	Therefore by theorem~\ref{local-thm}(1), this $G_K$-set 
	is unramified at all primes $\gp\mid\gd\gb^{-1}$. 
\end{proof}

\begin{proposition} \label{abelian}
	If $\E$ has an integral $\Lambda_{\scalA,\Pr}$-model
	and $\Pr$ is Chebotarev dense, then the action of
	$\G$ on $S$ factors through the abelianization of $\G$.
\end{proposition}
\begin{proof}
Let $\ringb$ be an integral $\Lambda_{\scalA,\Pr}$-model of $\E$.
For each prime $\gp\in\Pr$ we consider the completion
$\gpO_\gp$, and its fraction field $K_\gp$. Then we obtain an
$\LOp$-structure on the finite \'etale $K_\gp$-algebra $\E_\gp=
\E\otimes_K K_\gp$, and then
$\ringb\otimes_\scalA \gpO_\gp$ is an integral $\LOp$-model of $\E_\gp$.
Fixing an embedding $\Kb\to \Kbp$ for each $\gp$ we can view $\Gp$ as
a subgroup of $\G$. The finite \'etale $K_\gp$-algebra $\E_\gp$ then
corresponds to the $\Gp$-set that one gets by restricting the action
of $\G$ on $S$ to $\Gp$.

Now let $\Gb$ be the image of the action map $\G\to\Map(S,S)$.
Because $\Pr$ is Chebotarev dense, for each
$g\in\Gb$ there is a prime $\gp=\gp_g$ in $\Pr$ such that
\begin{enumerate}
	\item $\ringb$ is unramified at $\gp$,
	\item the image of $F_\gp\in \Gp/\Ip$ under the induced map $\Gp/\Ip\to\Gb$ is $g$.
\end{enumerate}
By proposition \ref{unramified}, the action of $g$ on $S$ is the same as the 
action of $\gp_g$ on $S$. But by the definition
of $\Lambda_{\scalA,\Pr}$-structure, the $\gp_g$ commute with each other.
Therefore $\Gb$ is abelian.
\end{proof}

\subsection{} \emph{Conductors.}
By class field theory, any continuous action of the abelianization of $\G$ on a finite
discrete set $T$ factors, by the Artin map,
through the ray class group $\Cl(c(T))$ for a minimal cycle $c(T)$
on $K$, which we call the conductor of $T$. 

\begin{lemma}\label{lem:equiv}
Assume that $\E$ has an integral $\Lambda_{\scalA,\Pr}$-model $\ringb$
and that the action of $\G$ on	$S$ factors through its abelianization.
Let $\gf$ be a cycle on $K$, and let $\gr$ be as in (\ref{equ:r-def}). 
Then following are equivalent:
\begin{enumerate}
	\item the action of $\G\times \Id_{\Pr}$ on $S$ factors through an action of $\DR_{\Pr}(\gf)$,
	\item $\gr$ divides $\gf$, and for each ideal $\gd\mid\gf$ we have $c(\gd S) \mid \gf\gd^{-1}$,
	\item $\gr$ divides $\gf$, and for each ideal $\gd\mid\gr$ we have $c(\gd S) \mid \gf\gd^{-1}$,
	\item the least common multiple $\lcm_{\gd\mid \gr} \big(\gd\, c(\gd S)\big)$ divides $\gf$.
\end{enumerate}
\end{lemma}
\begin{proof}
(1)$\Rightarrow$(2): 
To show $\gr\mid\gf$, we will show that for all $\gp\in\Pr$, we have $\gp^{n}S \subseteq \gp^{n+1} S$, where
$n=\ord_\gp(\gf)$. Using the decomposition
(\ref{DR-decomp}), we see $[\gp]^{n+1}\DR_\Pr(\gf)=[\gp]^n\DR_\Pr(\gf)$
and hence $[\gp]^n=[\gp]^{n+1}x$ for some $x\in\DR_\Pr(\gf)$. Then because the action of $\Id_P$
is assumed to factor through $\DR_{\Pr}(\gf)$, we have $\gp^n S = \gp^{n+1} xS \subseteq \gp^{n+1} S$.

Second, the condition $c(\gd S) \mid \gf\gd^{-1}$ of (2) is equivalent to the condition that
the action of $\G$ on $\gd S$ factors through the Artin map $\G\to\Cl(\gf\gd^{-1})$.
But this holds by the assumption (1) and the decomposition~(\ref{DR-decomp}).

(2)$\Rightarrow$(1): 
Consider an element $(\sigma,\ga)\in\G\times\Id_\Pr$. We will
show its action on $S$ depends only on its class in $\DR_\Pr(\gf)$.
First observe that by the assumption $c(\gd S) \mid \gf\gd^{-1}$, taken when $\gd=(1)$,
we have $c(S) \mid \gf$. This implies that the action of $\G$ factors through the Artin map 
$\G\to\Cl(\gf)$, and so it is enough to show that the action of $\ga$ depends only on its class 
$[\ga]\in\DR_\Pr(\gf)$.
Now put $\gd=\gcd(\ga,\gf)$ and $\ga'=\ga\gd^{-1}$. Then we have $\ga'\in\Id_\Pr(\gf)$. Therefore,
because of (\ref{DR-decomp}), it is enough to show that 
the action of $\ga'\in\Id_\Pr(\gf)$ on $\gd S$ depends only on its class $[\ga']\in\Cl_\Pr(\gf\gd^{-1})$.
In particular, it is enough to show
\begin{equation}
	\label{eq23498}
	\ga'\gd s = F_{\ga'}\gd s,	
\end{equation}
for all $s\in S$, where $F_{\ga'}\in\G$ is an element mapping to $[\ga']$ under the Artin map $\G\to\Cl(\gf)$.
To do this, it is enough to consider the case where $\ga'$ is a prime $\gp\in\Id_\Pr(\gf)$.
Now by our assumption $\gr\mid\gf$, we have $\gp\nmid\gr$ and hence $\gp$ acts bijectively on $S$.
Therefore by  theorem~\ref{local-thm}, we have
	$$
	\gp s = F_\gp s,
	$$
for all $s\in S$. This implies (\ref{eq23498}) and hence (1).

(2) $\Leftrightarrow$ (3): The implication (2) $\Rightarrow$ (3) is clear. So consider the other direction.
Given an ideal $\gd\mid\gf$, let $\gb$ denote $\gcd(\gd,\gr)$. Then by lemma \ref{gcd}, we have equivalences
	$$
	c(\gd S)  \mid \gf\gd^{-1}  \iff c(\gb S)\mid\gf\gd^{-1}
	\iff c(\gb S)\mid\gf\gd^{-1}\gd\gb^{-1} \iff c(\gb S)\mid\gf\gb^{-1}.
	$$
Therefore since $\gb$ is a divisor of $\gr$, the condition $c(\gd S)\mid \gf\gd^{-1}$ holds for all 
$\gd\mid\gf$ if it holds for all $\gd\mid\gr$. 

(3) $\Leftrightarrow$ (4): Clear.
\end{proof}

\begin{thm} 
\label{global-thm2}
Let $\gf$ be a cycle on $K$. Then the action of $\G\times \Id_{\Pr}$ on $S$ factors (necessarily uniquely)
through the map $\G\times \Id_{\Pr} \longmap \DR_{\Pr}(\gf)$ of (\ref{DR-recmap}) if and only if the following
hold:
\begin{enumerate}
	\item the action of $\G$ on $S$ factors through its abelianization,
	\item $\lcm_{\gd\mid \gr} \big(\gd\, c(\gd S)\big)$ divides $\gf$,
	\item $\E$ has an integral $\Lambda_{\scalA,\Pr}$-model.
\end{enumerate} 
If $\Pr$ is Chebotarev dense, then condition (1) can be removed.
\end{thm}

\begin{proof}
If (1)--(3) hold, then by lemma~\ref{lem:equiv},
the action of $\G\times \Id_{\Pr}$ on $S$ factors through $\DR_{\Pr}(\gf)$; and if $\Pr$ is Chebotarev 
dense, then (1) can be removed because, by proposition~\ref{abelian}, it follows from (3).

Let us now consider the converse direction.
Suppose that the $\G\times \Id_{\Pr}$-action on $S$ factors through $\DR_{\Pr}(\gf)$. Then (1)
holds because $\DR_\Pr(\gf)$ is commutative. Further, (3) implies (2) by lemma~\ref{lem:equiv}. Therefore it is
enough to show (3).

For each $\gp\in \Pr$, let $\ringb_\gp$ denote the maximal sub-$\Lambda_{\gpO_\gp,\gp}$-ring of $\E\otimes_K
K_\gp$ which is finite over $\gpO_\gp$. As mentioned in~\ref{rmk-group-alg}, it always exists. We will show that
$\ringb_\gp\otimes_{\gpO_\gp} K_\gp$ agrees with $E_\gp$. To do this, it is enough to show that $\E\otimes_K
K_\gp$ has an integral $\Lambda_{\gpO_\gp,\gp}$-model. So write $\gf=\gp^n\gf'$ with $n=\ord_\gp(\gf)$. Then
$[\gp^k] \in [\gp^n]\Cl(\gf')\subseteq \DR(\gf)$ for all $k\ge n$. This implies, by say~(\ref{eq:DR-mult}), that
the action of $\gp$ on $\bigcap_i\gp^i S=\gp^nS$ is given by the Artin symbol of $[\gp]\in \Cl(\gf')$, which by
our local result, theorem \ref{local-thm}, guarantees existence of an integral $\Lambda_{\gpO_\gp,\gp}$-model.

Now let $R$ denote the integral closure of $\scalA$ in $\E$, and let $\ringb$ denote the set of elements 
$a\in R$ such that the image of $a$ in $\E\otimes_K K_\gp$ lies in  $\ringb_\gp$ for all $\gp\in\Pr$. We will 
show that $\ringb$ is what we seek, an integral $\Lambda_{\scalA,\Pr}$-model for $E$.

For all $\gp\nmid \gf$, we are in the unramified case, and so $\ringb_\gp$ is $R\otimes_\scalA\localOp$, by
proposition \ref{unramified}. It follows that $\ringb$ is of finite index in $R$. Therefore, since $R$ is finite
and flat over $\scalA$ and we have $\E=R\otimes_\scalA K$, the same hold for $\ringb$. Further, $\ringb$ is
closed under all $\psi_\gq$ with $\gq\in\Pr$. Indeed it is enough to show $\psi_\gq(\ringb_\gp)\subseteq
\ringb_\gp$ for all $\gp,\gq\in\Pr$; and this holds because is a $\psi_\gq(\ringb_\gp)$
sub-$\Lambda_{\gpO_\gp,\gp}$-ring of $E_\gp$ which is finite over $\gpO_\gp$, and so it is contained in the
maximal one $\ringb_\gp$. Finally, for each $\gp\in \Pr$, the induced endomorphism $\psi_\gp$ of $\ringb$ is a
Frobenius lift, because $\ringb\otimes_\scalA \localOp$ is a $\Lambda_{\gpO_\gp,\gp}$-ring. Therefore $\ringb$
is an integral $\Lambda_{\scalA,\Pr}$-model for $E$. This establishes (3) and, hence, (1) and (2) as explained
above.
\end{proof}

\begin{corollary}
\label{cor:lambda-Galois-category-equiv}
If $\Pr$ is Chebotarev dense, there is a contravariant equivalence between the category of finite \'etale
$\Lambda_{\scalA,\Pr}$-rings over $K$ which admit an integral model and the category of finite discrete sets
with a continuous action of the profinite monoid $\lim_\gf \DR_{\Pr}(\gf)$. 
\end{corollary}

\part{Periodic loci and explicit class field theory}
\section{$\Lambda$-schemes}

Below $X$ will denote a flat $\scalA$-scheme. 

\subsection{} \emph{$\Lambda$-schemes.} Let $\gp$ be a maximal ideal of $\scalA$. As in the affine
case, the fiber $X\times_{\Spec\scalA}\Spec k(\gp)$ has a natural $k(\gp)$-scheme endomorphism
$F_{\gp}$ which is the identity map on the underlying topological space and such that for each
affine open subscheme $\Spec \ringb$, the induced the endomorphism of $\ringb$ is the affine Frobenius map
$x\mapsto x^{N(\gp)}$.

Let $\End_{\scalA}(X)$ denote the monoid of $\scalA$-scheme endomorphisms of $X$. An endomorphism
$\psi\in\End_{\scalA}(X)$ is said to be a \emph{Frobenius lift} at $\gp$ if the induced endomorphism on the fiber
$X\times_{\Spec \scalA}\Spec k(\gp)$ agrees with $F_{\gp}$. A \emph{$\Lambda_{\scalA,\Pr}$-structure} on a flat
$\scalA$-scheme $X$ is defined to be a set map $\Pr\to\End_{\scalA}(X)$, denoted $\gp\mapsto \psi_{\gp}$ such
that $\psi_{\gp}$ is a Frobenius lift at $\gp$ for each $\gp\in\Pr$ and such that
$\psi_{\gp}\circ\psi_{\gq}=\psi_{\gq}\circ\psi_{\gp}$ for all $\gp,\gq\in\Pr$. We will call an $\scalA$-scheme
with $\Lambda_{\scalA,\Pr}$-structure a \emph{$\Lambda_{\scalA,\Pr}$-scheme}. (When $X$ is not flat over
$\scalA$, this definition still makes sense; but as in the affine case, it is not well behaved. In general, one
should define it to be an action of the Witt vector monad $W_{\scalA,\Pr}^*$ as in the introduction
to~\cite{Borger:BGWV-II}. We will only consider $\Lambda$-structures on flat schemes in this paper; so the
simplified definition above is good enough here.)

For any ideal $\ga\in\Id_{\Pr}$ with prime factorization $\ga=\gp_1\cdots\gp_n$, 
let $\psi_{\ga}$ denote the composition $\psi_{\gp_1}\circ\cdots\circ \psi_{\gp_n}$.
It is independent of the order of the factors because the operators $\psi_{\gp}$ commute with each other.

A morphism $X\to Y$ of $\Lambda_{\scalA,\Pr}$-schemes is a morphism $f\colon X\to Y$ of $\scalA$-schemes such
that $f\circ\psi_\gp=\psi_\gp\circ f$, for all $\gp\in\Pr$. In this way, $\Lambda_{\scalA,\Pr}$-schemes form a
category.

\subsection{}\emph{Examples.} 
\label{subsec:toric-lambda-schemes}
The multiplicative group $X=\Gm$ over $\scalA$ has a $\Lambda_{\scalA,\Pr}$-structure given by
$\psi_\gp(x)=x^{N(\gp)}$. This extends uniquely to $\Lambda_{\scalA,\Pr}$-structures on $\A^1$ and $\P^1$. More
generally, projective $n$-space $\P^n$ has a $\Lambda_{\scalA,\Pr}$-structure where $\psi_\gp$ raises the
homogeneous coordinates to the $N(\gp)$ power.

Any product of $\Lambda_{\scalA,\Pr}$-schemes is again a $\Lambda_{\scalA,\Pr}$-scheme, where the
$\psi$-operators act componentwise. In a similar way, coproducts of $\Lambda_{\scalA,\Pr}$-schemes are $\Lambda_{\scalA,\Pr}$-schemes.

\section{Periodic $\Lambda$-schemes}

Let $\gf$ be a cycle on $K$. 

\subsection{} \emph{Periodic $\Lambda$-schemes.}
We will say that a $\Lambda_{\scalA,\Pr}$-scheme $X$ is \emph{$\gf$-periodic} if
for all $\gf$-equivalent ideals $\ga,\gb\in\Id_{\Pr}$,
the two maps $\psi_{\ga},\psi_{\gb}\colon X \to X$ are equal---in other words, the monoid map
$\Id_\Pr\to\End_A(X)$ factors through $\DR_\Pr(\gf)$.

\subsection{} \emph{Examples.}
Suppose that $\scalA=\Z$, $\Pr=\mx_\Q$, and
$\gf=(n)\infty$ with $n$ positive. Then $\gf$-periodicity means that $\psi_{(a)}=\psi_{(a+n)}$ for all
integers $a\geq 1$. In other words, the sequence of Frobenius lifts $\psi_{(a)}$ is periodic in $a\geq 1$ with
period dividing $n$, which is the reason for the name. Representation rings of finite groups, with their
usual $\lambda$-ring structure in algebraic K-theory, are periodic. In fact, periodicity was first introduced in 
this context by Davydov~\cite{Davydov:periodic-lambda-rings}.

On the other hand, when $\scalA$ is general but $\gf$ is $(1)$, then $\gf$-periodicity means that $\psi_{\ga}$
depends only on the class of $\ga$ in the class group $\Cl(1)$. If $\gf$ is instead the product of all
real places, then it means that $\psi_{\ga}$ depends only on the class of $\ga$ in the narrow class group
$\Cl(\gf)$. In particular, if either of these class groups is trivial, then every $\psi_\ga$ is the
identity map, and so any (flat) $\scalA$-scheme has at most one $\Lambda_{\scalA,\Pr}$-structure with that type
of periodicity. 

If we are in the intersection of the two cases above, a $(1)$-periodic $\Lambda_{\Z,\Pr}$-ring is
just a $\lambda$-ring in which all the Adams operations are the identity. 
Elliott~\cite{Elliott:binomial-rings} has proved that this is equivalent to being a binomial ring, 
a notion which dates back to Berthelot's expos\'e in SGA6~\cite{SGA6}, p.\ 323.

\begin{proposition}
\label{pro:periodic-structure2}
Let $X$ be a separated flat $\gf$-periodic $\Lambda_{\scalA,\Pr}$-scheme of finite type.
\begin{enumerate}
	\item If $\Pr$ is infinite, then $X$ is affine, reduced, and quasi-finite over $\scalA$, with
		\'etale generic fiber $X_K$. 
	\item If $\Pr$ is Chebotarev dense, then the generic fiber's function algebra
		$\sheafO(X_K)$ is a product of subextensions of the ray class field $K(\gf)$.
\end{enumerate} 
\end{proposition}
\begin{proof}

(1): Let us first show that $X$ is reduced and quasi-finite over $\scalA$. By periodicity, for every $\gp\in\Pr$
satisfying $\gp\nmid\gf$, the Frobenius lift $\psi_{\gp}$ on $X$ is an automorphism of finite order. Therefore
the Frobenius map $F_\gp$ on the fiber $k(\gp)\otimes X$ is an automorphism of finite order, and so the fiber is
both geometrically reduced and finite over $k(\gp)$. Since $\Pr$ is infinite and since the set of prime ideals
of $\scalA$ with geometrically reduced fibers forms a constructible subset of $\Spec\scalA$, by EGA IV (9.7.7)
\cite{EGA-no.28}, the generic fiber $X\times_{\Spec \scalA}\Spec K$ must be geometrically reduced. Similarly,
since infinitely many fibers are finite, the generic fiber is also finite, by EGA IV (9.2.6.2)~\cite{EGA-no.28}.
We now use the flatness of $X$ over $\scalA$ to pass from the generic fiber to all of $X$. It is clear that
flatness implies $X$ is reduced. For quasi-finiteness, apply EGA IV (14.2.4) \cite{EGA-no.28}.

Therefore $X$ is reduced and is quasi-finite over $\scalA$. Affineness then follows from Zariski's Main Theorem,
as we now explain. Let $\ringd$ denote the integral closure of $\scalA$ in $\sheafO(X)$. Let $X'$ denote $\Spec
\ringd$, called the normalization of $\scalA$ in $X$ in the Stacks Project
\cite[\href{http://stacks.math.columbia.edu/tag/035H}{Tag 035H}]{stacks-project}. Then $\sheafO(X)$ is reduced
and is flat over $\scalA$, and hence so is $\ringd$. It is also finite over $\scalA$, by
\cite[\href{http://stacks.math.columbia.edu/tag/03GR}{Tag 03GR}]{stacks-project}. Since $X$ is quasi-finite,
the canonical map $X\to X'$ is an open immersion by Zariski's Main Theorem
\cite[\href{http://stacks.math.columbia.edu/tag/03GW}{Tag 03GW}]{stacks-project}. Finally, since $X'$ is finite
flat over $\scalA$, which is a Dedekind domain, its Krull dimension is $1$ and hence its open subscheme $X$ must
be affine \cite[\href{http://stacks.math.columbia.edu/tag/09N9}{Tag 09N9}]{stacks-project}.

Thus we can write $X=\Spec \ringb$, where $\ringb$ is an $\gf$-periodic $\Lambda_{\scalA,\Pr}$-ring which is
reduced and is flat, quasi-finite, and of finite type over $\scalA$. It follows that $K\otimes_{\scalA}\ringb$
is \'etale over $K$.

(2): It follows from statement (1) that there is a nonzero element $t\in \scalA$ such that $\ringb[1/t]$ is a
finite product $\prod_i \ringd_i$, where each $\ringd_i$ is the integral closure of $\scalA[1/t]$ in a finite
extension $L_i$ of $K$: Indeed, since $\ringb$ is of finite type over $\scalA$ and since
$\ringb\otimes_{\scalA}K$ is finite over $K$, there is an element $t\in\scalA$ such that $\ringb[1/t]$ is finite
flat over $\scalA[1/t]$. Since $\ringb$ is reduced, so is $\ringb[1/t]$, and hence the discriminant ideal of
$\ringb[1/t]$ over $\scalA[1/t]$ is nonzero. Then by scaling $t$ so that it lies in the discriminant ideal, and
is nonzero, we may assume that $\ringb[1/t]$ is finite \'etale over $\scalA[1/t]$. It then follows that
$\ringb[1/t]$ is the integral closure of $\scalA[1/t]$ in $K\otimes_{\scalA}\ringb$ and is hence of the required
form.

Now let $\Pr_t$ denote the set of primes $\gp$ in $\Pr$ that do not divide $t$. Then we can consider $\Pr_t$ as
a set of primes of $\scalA[1/t]$; and since $\Pr$ is Chebotarev dense, so is $\Pr_t$. Applying proposition
\ref{abelian} to the $\Lambda_{\scalA[1/t],\Pr_t}$-ring $\ringb[1/t]$, we see each field 
$K\otimes_\scalA \ringd_i$ is an abelian extension of $K$. Since the Frobenius elements act on each $\ringd_i$
with period $\gf$, the conductor of $K\otimes_\scalA \ringd_i$ divides $\gf$. \end{proof}

\subsection{} \emph{Ray class algebras.}
\label{subsec:ray-class-lambda-rings}
These are $\Lambda$-ring analogues of the ray class fields of $K$. 

We can view $\DR_{\Pr}(\gf)$ as a pointed $\DR_{\Pr}(\gf)$-set: the distinguished point is the identity element,
and the action is translation. By theorem~\ref{global-thm2}, the corresponding $\Lambda_{\scalA,\Pr}$-ring over $K$ has an integral model. Define $\rcl{\scalA,\Pr}{\gf}$, the \emph{ray class algebra of conductor $\gf$},
to be the maximal $\Lambda_{\scalA,\Pr}$-order in this $K$-algebra. Thus we have
	$$ 
	K\otimes_{\scalA} \rcl{\scalA,\Pr}{\gf} 
		= \prod_{\substack{\gd\in\Id_{\Pr}\\\gd\mid\gf_{\fin}}} K(\gf\gd^{-1}), 
	$$ 
where $K(\gf\gd^{-1})$ is the ray class
field of $K$ with conductor $\gf\gd^{-1}$. Under this identification, the map 
$\beta\colon \rcl{\scalA,\Pr}{\gf}\to\Kb$ coming from the distinguished point of $\DR_{\Pr}(\gf)$
is the projection to the component with $\gd=(1)$.

In our previous paper \cite{Borger-deSmit:Integral-models}, we considered the case where $\scalA=\Z$ and $\Pr$ is
all maximal ideals. There we showed that the ray class algebra of conductor $(n)\infty$ is
$\Z[x]/(x^n-1)$, or more naturally, the group ring on the cyclic group $\mu_n(\Kb)$ of $n$-th roots of unity
in $\Kb$.

Observe that the ray class algebra is not usually a domain, and in particular the map
$\beta\:\rcl{\scalA,\Pr}{\gf}\to K(\gf)$ to the ray class field is not usually injective. 
Also, unlike the ray class field, it depends not only on $\gf$ but also on $\scalA$ and $\Pr$. On the
other hand, $K\otimes_{\scalA} \rcl{\scalA,\Pr}{\gf}$ is independent of $\scalA$.

If $\Pr$ is Chebotarev dense, the ray class algebra $\rcl{\scalA,\Pr}{\gf}$ satisfies the following
maximality property: if $\ringd$ is a reduced finite flat $\gf$-periodic $\Lambda_{\scalA,\Pr}$-ring equipped with a map
$\alpha\colon \ringd\to\Kb$ of $\scalA$-algebras, then there is a unique map $\varphi\colon \ringd\to \rcl{\scalA,\Pr}{\gf}$ of
$\Lambda_{\scalA,\Pr}$-rings making the following diagram commute:
	$$
	\xymatrix{
	& \Kb \\
	\ringd \ar^{\alpha}[ur] \ar@{-->}^{\varphi}[rr] 
		& & \rcl{\scalA,\Pr}{\gf}. \ar_{\beta}[ul]
	}
	$$
This is simply because, under the anti-equivalence with $\DR_{\Pr}(\gf)$-sets, $\rcl{\scalA,\Pr}{\gf}$ 
corresponds to $\DR_{\Pr}(\gf)$, which is the free $\DR_{\Pr}(\gf)$-set on one generator.

\subsection{} \emph{Change of $\gf$.}
Suppose $\gf'=\gf\ga$. Then the maps $\DR_\Pr(\gf')\to \DR_{\Pr}(\gf)$ and $\DR_\Pr(\gf) \longmap \DR_\Pr(\gf')$ 
of~\ref{subsec:DR-change-of-f} induce an inclusion
	\begin{equation}
	u\:\rcl{\scalA,\Pr}{\gf} \longmap \rcl{\scalA,\Pr}{\gf'}
	\end{equation}
and a surjection
	\begin{equation}
	v\:\rcl{\scalA,\Pr}{\gf'} \longmap \rcl{\scalA,\Pr}{\gf}	
	\end{equation} 
of $\Lambda_{\scalA,\Pr}$-rings, and
the compositions $u\circ v$ and $u\circ v$ agree with the two $\psi_\ga$ endomorphisms.

In the case where $K=\Q$, $\Pr$ is all maximal ideals, $\gf=(n)\infty$, and $\gf'=(n')\infty$, these maps can be 
identified with the maps on group rings corresponding to the inclusion $\mu_{n}(\Kb)\subseteq \mu_{n'}(\Kb)$, in 
the case of $u$, and the $n'/n$-th power map $\mu_{n'}(\Kb)\to \mu_{n}(\Kb)$, in the case of $v$.

\section{Periodic loci and abelian extensions}

Let $\gf$ be a cycle on $K$. Let $X$ be a separated (flat) $\Lambda_{\scalA,\Pr}$-scheme. 

\subsection{} \emph{Periodic locus $\per{X}{\gf}$.}
\label{subsec:periodic-loci}
Define the $\gf$-periodic locus $\per{X}{\gf}$ of $X$ to be the scheme-theoretic intersection
	$$
	\per{X}{\gf} = \bigcap_{\substack{\ga,\gb\in\Id_{\Pr}\\\ga\sim_{\gf}\gb}} X(\psi_{\ga}=\psi_{\gb}),
	$$
where $X(\psi_{\ga}=\psi_{\gb})$ denotes the equalizer of the two maps $\psi_{\ga},\psi_{\gb}\colon X\to X$.  
Since $X$ is separated, $\per{X}{\gf}$ is a closed subscheme. 
The functor $\per{X}{\gf}$ represents is
	$$
	\per{X}{\gf}(C) = \{x\in X(C): \psi_{\ga}(x)=\psi_{\gb}(x) \text{ for all } \ga,\gb\in\Id_{\Pr} 
	\text{ with } \ga\sim_{\gf}\gb\}.
	$$
We emphasize that $\per{X}{\gf}$ depends on $\Pr$ as well as $\gf$, although the notation does not reflect this.

Observe that $X(\gf)$ is functorial in $X$. It also behaves well under change of $\gf$:
if $\gf'$ is another cycle and $\gf\mid\gf'$, then we have
	$$
	\per{X}{\gf} \subseteq \per{X}{\gf'}.
	$$

Finally, let $\perfl{X}{\gf}$ denote the maximal $\scalA$-flat subscheme of $\per{X}{\gf}$. 
It is the closed subscheme of $\per{X}{\gf}$ defined by the ideal sheaf of $\scalA$-torsion elements.
Although the actual periodic locus $\per{X}{\gf}$ is more fundamental than $\perfl{X}{\gf}$,
for the purposes of this paper it is enough to consider $\perfl{X}{\gf}$, and doing so will allow us to
avoid the subtleties of $\Lambda$-rings with torsion.

\begin{proposition}
\label{pro:preimage-of-periodic-locus}
For any ideal $\ga\in\Id_{\Pr}$, we have the subscheme inclusion
	\begin{equation} \label{eq:382}
		\per{X}{\ga\gf}\subseteq \psi_{\ga}^{-1}(\per{X}{\gf}).
	\end{equation}
\end{proposition}
\begin{proof}
Let $x$ be a point of $\per{X}{\ga\gf}$ with coordinates in some ring $R$.
Let $\gb$ and $\gc$ be $\gf$-equivalent ideals in $\Id_{\Pr}$.
Then we have $\ga\gb\sim_{\ga\gf}\ga\gc$ and hence
	$$
	\psi_{\gb}\big(\psi_{\ga}(x)\big)=\psi_{\ga\gb}(x)=\psi_{\ga\gc}(x)=\psi_{\gc}\big(\psi_{\ga}(x)\big).
	$$
It follows that $\psi_{\ga}(x)\in \per{X}{\gf}$, and this implies~(\ref{eq:382}).
\end{proof}

\subsection{} \emph{Torsion locus $X[\gf]$.}
Write $\gf=\gf_{\fin}\gf_{\infty}$. Then we define the \emph{$\gf$-torsion} locus by
	$$
	X[\gf] = \psi_{\gf_{\fin}}^{-1}(\per{X}{\gf_{\infty}}).
	$$
It follows from proposition~\ref{pro:preimage-of-periodic-locus} that we have an inclusion of subschemes
	\begin{equation}
	\label{eq:periodic-in-torsion}
		\per{X}{\gf} \subseteq X[\gf].
	\end{equation}
For example, if $\scalA=\Z$, $\Pr=\mx_\Q$, and $\gf=(n)\infty$, then this inclusion
is an equality because both sides are $\mu_n$. If however $\gf=(n)$, then the $\gf$-torsion locus
is again $\mu_n$, but the $\gf$-periodic locus $\Gm(\gf)$ is $\mu_m$, where $m={\gcd(2,n)}$. So the 
containment~(\ref{eq:periodic-in-torsion}) can be far from an equality.

\vspace{3mm}
Summing up the results above with the previous section, we have the following:

\begin{thm}
\label{thm:periodic-locus3} 
If $X$ is of finite type over $\scalA$ and $\Pr$ is Chebotarev dense, then we have the following:
\begin{enumerate}
	\item The flat $\gf$-periodic locus $\perfl{X}{\gf}$ is a closed $\gf$-periodic 
	sub-$\Lambda_{\scalA,\Pr}$-scheme of $X$, and it is the maximal flat closed subscheme with these properties. 
	\item We have $\perfl{X}{\gf}=\Spec \ringb$, where $\ringb$ is a finitely 
	generated $\scalA$-algebra, and
	$K\otimes_\scalA \ringb$ is a finite product of abelian extensions of $K$ of conductor dividing $\gf$.	
	\item If $X$ is proper, then $\ringb$ is a finite $\scalA$-algebra.
\end{enumerate}
\end{thm}

\begin{proof}
(1) By functoriality, the subscheme $\per{X}{\gf}$ is stable under the
operators $\psi_{\gp}$, for all $\gp\in\Pr$. Again by functoriality, $\perfl{X}{\gf}$ is also stable under
them. Since the $\psi_{\gp}$ are Frobenius lifts on $X$, so are the endomorphisms they induce on the closed
subscheme $\perfl{X}{\gf}$. Since $\perfl{X}{\gf}$ is flat, this defines a $\Lambda_{\scalA,\Pr}$-structure on
$\perfl{X}{\gf}$. It is obviously $\gf$-periodic. Maximality is also clear.

(2) This follows from proposition \ref{pro:periodic-structure2}.

(3) When $X$ is proper, so is $\perfl{X}{\gf}$, since it is a closed subscheme of $X$. It therefore must be 
finite over $\scalA$ because it is flat and generically finite.
\end{proof}

\subsection{} \emph{Conjecture.}
\label{conj:points-exist}
Let $\gr$ denote the product of all the real places.
If $X$ is proper and nonempty, we conjecture that $\per{X}{\gr}$ is nonempty.

When $K=\Q$, this was proved in~\cite{Borger:LRFOE}. It is possible, however, to give an easier argument that 
avoids the deep theorems in \'etale cohomology and $p$-adic Hodge theory used there. One would expect this 
argument to go through for general $K$.

\subsection{} \emph{Computing the periodic locus.}
Is there an algorithm to find equations describing $X(\gf)$, given equations for $X$ and formulas for the
Frobenius lifts $\psi_\gp$? Without such an algorithm, our approach to generating abelian extensions would not be
so explicit. But it also might be an indication that, for theoretical purposes,
there really is more freedom in this approach. 

\section{Interlude on periodic Witt vectors}

In this section, we define periodic Witt vectors and show how they recover the ray
class algebras of the previous section. It will not be used elsewhere in the paper.

\subsection{} \emph{$\Pr$-typical Witt vector rings $W_{\scalA,\Pr}(\ringr)$.} Let us review the generalized
Witt vector rings as defined in the first section of
\cite{Borger:BGWV-I}. Let $\ringr$ be a flat $\scalA$-algebra. Then
the monoid $\Id_{\Pr}$ acts on the product $\scalA$-algebra $\ringr^{\Id_{\Pr}}$ (the so-called ghost ring)
by translation in the exponent. Explicitly,
if $x_{\ga}$ denotes the $\ga$-th component of 
a vector $x\in \ringr^{\Id_{\Pr}}$, then $\psi_{\gb}\colon \ringr^{\Id_{\Pr}}\to \ringr^{\Id_{\Pr}}$
is defined by the formula
	$$
	\big(\psi_{\gb}(x)\big)_{\ga} = x_{\ga\gb}.
	$$

Now consider the set of sub-$\scalA$-algebras $\ringd\subseteq \ringr^{\Id_{\Pr}}$ such that $\ringd$ is
taken to itself by the action of $\Id_{\Pr}$ and such that for each prime $\gp\in\Pr$, the induced endomorphism
$\psi_{\gp}\colon \ringd\to \ringd$ is a Frobenius lift at $\gp$. An elementary argument shows that this
collection of subrings has a maximal element $W_{\scalA,\Pr}(\ringr)$. It is called the ring of $\Pr$-typical
Witt vectors with entries in $\ringr$. It recovers the usual $p$-typical Witt vector functor (restricted to
torsion-free rings) when $\scalA$ is $\Z$ and $\Pr$ consists of the single maximal ideal $p\Z$; it recovers the
big Witt vector functor when instead $\Pr$ consists of all maximal ideals of $\Z$.

This construction is functorial in $\ringr$, and one can show that the functor $W_{\scalA,\Pr}$ is representable
by a flat $\scalA$-algebra $\Lambda_{\scalA,\Pr}$. (Incidentally, this shows that $W_{\scalA,\Pr}$ extends to a
functor on all $\scalA$-algebras, namely the one represented by $\Lambda_{\scalA,\Pr}$. Thus we can extend the
theory of $\Lambda$-structures and Witt vectors to $\scalA$-algebras with torsion, but we will not need this
generality here.) The endomorphisms $\psi_{\ga}$ of the functor $W_{\scalA,\Pr}$ induce endomorphisms of
$\Lambda_{\scalA,\Pr}$. They are in fact Frobenius lifts, and hence $\Lambda_{\scalA,\Pr}$ is a
$\Lambda_{\scalA,\Pr}$-ring.

\subsection{} \emph{Universal property of Witt vector rings.}
\label{subsec:witt-universal-property} 
The Witt vector functor is the right adjoint of the forgetful functor from $\Lambda_{\scalA,\Pr}$-rings to
$\scalA$-algebras. Let us spell out the universal property for future reference.
Let $\ringr$ be a flat $\scalA$-algebra, let $\ringd$ be a flat $\Lambda_{\scalA,\Pr}$-ring, and let 
$\varphi\colon\ringd\to \ringr$ be an $\scalA$-algebra map. Then $\varphi$ lifts to a unique 
$\Lambda_{\scalA,\Pr}$-ring map $\tilde{\varphi}$ to the Witt vector ring:
	$$
	\xymatrix{
	\ringd\ar_{\varphi}[dr]\ar@{-->}_-{\exists!}^-{\tilde{\varphi}}[rr] 
		& & W_{\scalA,\Pr}(\ringr) \ar^{x\mapsto x_{(1)}}[dl] \\
	& \ringr.
	}
	$$
Indeed, it lifts to a unique $\Id_\Pr$-equivariant morphism $\ringd\to \ringr^{\Id_\Pr}$, given by $a\mapsto x$,
where $x_\ga=\varphi(\psi_\ga(a))$. It remains to check that the image $S$ of this map lies in
$W_{\scalA,\Pr}(\ringr)$. But $S$ is torsion free, as a subalgebra of $\ringr^{\Id_\Pr}$; it has an action of
$\Id_\Pr$, as the image of an equivariant map of rings with an $\Id_\Pr$-action; and it satisfies the Frobenius
lift condition because it is an $\Id_\Pr$-equivariant quotient of $\ringd$, which satisfies the Frobenius lift
condition. Therefore $S$ is contained in $W_{\scalA,\Pr}(\ringr)$ by the maximality property of
$W_{\scalA,\Pr}(\ringr)$.

\subsection{} \emph{$\gf$-periodic Witt vector rings $\peru{W}{\gf}(\ringr)$.} 
Let $\gf$ be a cycle on $K$. We define the set
of $\gf$-periodic Witt vectors with entries in an $\scalA$-algebra $\ringr$ as follows
	\begin{equation}
		\peru{W}{\gf}_{\scalA,\Pr}(\ringr) := 
			\{x\in W_{\scalA,\Pr}(\ringr): \psi_{\ga}(x)=\psi_{\gb}(x) \text{ whenever } \ga\sim_{\gf}\gb\}.
	\end{equation}
In other words, if we view the functor $W_{\scalA,\Pr}$ as a scheme, then $\peru{W}{\gf}_{\scalA,\Pr}(\ringr)$ is the set of 
$\ringr$-valued points on its $\gf$-periodic locus.

When $\ringr$ is flat, the periodic Witt vectors can be described simply in terms of their ghost components:
	\begin{equation}
	\label{def:periodic-witt}
		\peru{W}{\gf}_{\scalA,\Pr}(\ringr) = \{x\in W_{\scalA,\Pr}(\ringr): x_{\ga}=x_{\gb} \text{ whenever } \ga\sim_{\gf}\gb\}.
	\end{equation}
Indeed, this follows from the implication $\ga\sim_{\gf}\gb \Rightarrow \ga\gc\sim_{\gf}\gb\gc$.
In other words, we have
	\begin{equation}
	\label{eq:periodic-witt-intersection}
		\peru{W}{\gf}_{\scalA,\Pr}(\ringr) = W_{\scalA,\Pr}(\ringr) \cap \ringr^{\DR_{\Pr}(\gf)},
	\end{equation}
as subrings of the ghost ring $\ringr^{\Id_{\Pr}}$.

\subsection{} \emph{Example.}
Suppose $\scalA=\Z$, $\Pr=\mx_\Q$, and $\gf=(n)\infty$ where $n\geq 1$. Then an
$\gf$-periodic Witt vector with entries in a torsion-free ring $\ringr$ is just a big Witt vector whose ghost
components are periodic with period dividing $n$. This is the reason for the name. For example, if $\zeta_n$ is
an $n$-th root of unity, then the Teichm\"uller element 
$$
[\zeta_n] := \langle \zeta_n,\zeta_n^2,\zeta_n^3,\dots \rangle
$$ 
is $n\infty$-periodic. (This is even true when
$\ringr$ is not torsion free, by functoriality and because the universal ring with an $n$-th root of unity is
$\Z[x]/(x^n-1)$, which is torsion free.)

\begin{proposition}
\label{pro:w-per-is-sub-lambda-ring}
$\peru{W}{\gf}_{\scalA,\Pr}(\ringr)$
is an $\gf$-periodic sub-$\Lambda_{\scalA,\Pr}$-ring of $W_{\scalA,\Pr}(\ringr)$, for
any flat $\scalA$-algebra $\ringr$.
\end{proposition}

\begin{proof}
It is clearly a sub-$\scalA$-algebra of $W_{\scalA,\Pr}(\ringr)$. 
It is also preserved by all $\psi_{\gp}$ operators ($\gp\in\Pr$) because of the implication
$\ga\sim_{\gf}\gb \Rightarrow \gp\ga\sim_{\gf} \gp\gb$. The family of operators $\psi_\ga$ is
also $\gf$-periodic on $\peru{W}{\gf}_{\scalA,\Pr}(\ringr)$ by definition.

So all that remains is to check the Frobenius lift property.
For $\gp\in\Pr$ and any $x\in \peru{W}{\gf}_{\scalA,\Pr}(\ringr)$, we have
	$$
	\psi_{\gp}(x)-x^{N(\gp)}\in \gp W_{\scalA,\Pr}(\ringr) \cap \peru{W}{\gf}_{\scalA,\Pr}(\ringr).
	$$
Thus it is enough to show that the containment
	\begin{equation}
		\label{eq:per-Frob-lift}
		\gp\peru{W}{\gf}_{\scalA,\Pr}(\ringr) \subseteq \gp W_{\scalA,\Pr}(\ringr) \cap \peru{W}{\gf}_{\scalA,\Pr}(\ringr)
	\end{equation}
is an equality. This follows from the diagram
	$$
	\xymatrix{
	0 \ar[r]
		& \peru{W}{\gf}(\ringr) \ar[r]
		& W(\ringr) \ar^-{\gamma}[r]
		& \prod_{\ga\sim_{\gf}\gb} \ringr \\
	0 \ar[r] 
		& \gp\otimes_{\scalA}\peru{W}{\gf}(\ringr) \ar[r]\ar@{>->}[u]
		&  \gp\otimes_{\scalA}W(\ringr) \ar^-{\gp\otimes\gamma}[r]\ar@{>->}[u]
		&  \gp\otimes_{\scalA}\prod_{\ga\sim_{\gf}\gb} \ringr\ar@{>->}[u] \\
	}
	$$
of exact sequences, where $\gamma(x)=(\dots,x_{\ga}-x_{\gb},\dots)$.
The vertical arrows are the evident multiplication maps $a\otimes b\mapsto ab$;
they are injective because all modules on the top row are flat.
This plus the exactness of the bottom row implies that (\ref{eq:per-Frob-lift}) is an equality.
\end{proof}

\subsection{} \emph{Remark: Plethysic algebra.}
The formal concepts above can be expressed in the language of plethystic algebra \cite{Borger-Wieland:PA}. Let
$\Lambda_{\scalA,\Pr}$ be the $A$-algebra representing the functor $W_{\scalA,\Pr}$, and let
$\peru{\Lambda}{\gf}_{\scalA,\Pr}$ be the one representing the functor $\peru{W}{\gf}_{\scalA,\Pr}$. Then the
inclusion of functors $\peru{W}{\gf}_{\scalA,\Pr} \subseteq W_{\scalA,\Pr}$ induces a surjection
$\Lambda_{\scalA,\Pr}\to \peru{\Lambda}{\gf}_{\scalA,\Pr}$. It is not hard to show that this has the structure
of a morphism of $\scalA$-plethories and an action of $\peru{\Lambda}{\gf}_{\scalA,\Pr}$ is the same as an
$\gf$-periodic $\Lambda_{\scalA,\Pr}$-structure. It follows that the forgetful functor from $\gf$-periodic
$\Lambda_{\scalA,\Pr}$-rings to all $\Lambda_{\scalA,\Pr}$-rings has both a left and a right adjoint. The right
adjoint outputs the $\gf$-periodic elements of the given $\Lambda_{\scalA,\Pr}$-ring.
This provides one approach to
$\Lambda$-structures and Witt vectors that works smoothly in the presence of torsion.

\subsection{} \emph{Universal property of periodic Witt vector rings.}
\label{subsec:periodic-witt-univ-prop}
Let $\ringr$ be a flat $\scalA$-algebra,
let $\ringd$ be a flat $\gf$-periodic $\Lambda_{\scalA,\Pr}$-ring, and let
$\varphi\colon \ringd\to \ringr$ be an $\scalA$-algebra map. Then $\varphi$
lifts to a unique $\Lambda_{\scalA,\Pr}$-ring
map to the $\gf$-periodic Witt vector ring:
	$$
	\xymatrix{
	\ringd\ar_{\varphi}[dr]\ar@{-->}_-{\exists!}^-{\tilde{\varphi}}[rr] 
		& & \peru{W}{\gf}_{\scalA,\Pr}(\ringr) \ar^{x\mapsto x_{(1)}}[dl] \\
	& \ringr.
	}
	$$
Indeed, by the universal property of Witt vectors (\ref{subsec:witt-universal-property}),
it lifts to a unique $\Lambda_{\scalA,\Pr}$-map
$\ringd\to W_{\scalA,\Pr}(\ringr)$. But since $\ringd$ is $\gf$-periodic, the image 
is contained in $\peru{W}{\gf}_{\scalA,\Pr}(\ringr)$.

\begin{proposition}
	Let $\Kb$ be a separable closure of $K$, and let $\intO$ denote the integral closure of $\scalA$ in $\Kb$. 
	Let
	$\rcl{\scalA,\Pr}{\gf}$ denote the ray class $\Lambda_{\scalA,\Pr}$-ring of conductor $\gf$, as defined in
	\ref{subsec:ray-class-lambda-rings}. Then we have an isomorphism
		\begin{equation}
			\tilde{\beta}\:\rcl{\scalA,\Pr}{\gf} \longrightisomap \peru{W}{\gf}_{\scalA,\Pr}(\intO)
		\end{equation}
	of $\Lambda_{\scalA,\Pr}$-rings, where the map $\tilde{\beta}$ is the 
	lift, in the sense of \ref{subsec:periodic-witt-univ-prop}, of the projection 
	$\beta\:\rcl{\scalA,\Pr}{\gf}\to\intO$ defined in \ref{subsec:ray-class-lambda-rings}.
\end{proposition}

\begin{proof}
The $\Lambda_{\scalA,\Pr}$-rings $\rcl{\scalA,\Pr}{\gf}$ and $\peru{W}{\gf}_{\scalA,\Pr}(\intO)$ are
characterized by the same universal property, except that the one for $\rcl{\scalA,\Pr}{\gf}$ is restricted to
algebras $\ringd$ that are reduced and finite flat over $\scalA$. So it is enough to show that
$\peru{W}{\gf}_{\scalA,\Pr}(\intO)$ is itself reduced and finite flat over $\scalA$. Since it is a subalgebra of
$\prod_{\DR_{\Pr}(\gf)}\intO$ by definition, it is reduced and flat. So it is enough to show that it is finite
over $\scalA$.

First, observe that we have
	\begin{equation}
	\label{eq:24}
		\colim_L \peru{W}{\gf}_{\scalA,\Pr}(\tO{L})
			= \peru{W}{\gf}_{\scalA,\Pr}(\intO),
	\end{equation}
where $L$ runs over the finite extensions of $K$ contained in $\Kb$. Indeed, for any
$x\in\peru{W}{\gf}_{\scalA,\Pr}(\intO)$, let $L$ denote the extension $K(\dots,x_\ga,\dots)$, where the $x_\ga$ are
the (ghost) components of $x$. It follows that $x\in W_{\scalA,\Pr}(L)\cap W_{\scalA,\Pr}(\intO)$. Because $W_{\scalA,\Pr}$ is representable, we also have
	$$
	W_{\scalA,\Pr}(L)\cap W_{\scalA,\Pr}(\intO) = W_{\scalA,\Pr}(L\cap\intO),
	$$
and hence $x\in W_{\scalA,\Pr}(L\cap\intO)$. 
Because $x$ is also $\gf$-periodic,
we have 
$$
x\in\peru{W}{\gf}_{\scalA,\Pr}(\tO{L}),
$$ 
as desired. It remains to show that $L$
is a finite extension of $K$. This holds because $x$ is $\gf$-periodic and $\DR_{\Pr}(\gf)$ is finite, and so
$x$ has only finitely many distinct ghost components. This proves~(\ref{eq:24}).

Therefore it is enough to prove that $\peru{W}{\gf}_{\scalA,\Pr}(\tO{L})$ has bounded rank as $L$ runs over all
finite extensions of $K$. By proposition \ref{pro:w-per-is-sub-lambda-ring},
$\peru{W}{\gf}_{\scalA,\Pr}(\tO{L})$ is a flat reduced $\gf$-periodic $\Lambda_{\scalA,\Pr}$-ring. Being a
subring of $\prod_{\DR_{\Pr}(\gf)}\tO{L}$, it is also finitely generated as an $\scalA$-module. Therefore by
theorem \ref{global-thm2}, it is contained in the product ring $\prod_{\gd} \tO{K(\gf\gd^{-1})}$, where $\gd$
runs over $\Id_{\Pr}$ with $\gd\mid\gf_{\fin}$. Therefore its rank as $L$ varies is bounded. 
\end{proof}

\subsection{} \emph{Example.}
\label{subsec:cyclotomic-ray-class-algebra}
If $\scalA=\Z$, $P=\mx_\Q$, and $\gf=(n)\infty$ with $n\geq 1$, then for any primitive
$n$-th root of unity $\zeta_n$, there is an isomorphism
	$$
	\Z[x]/(x^n-1) \longrightisomap \peru{W}{n\infty}_{\Z,\Pr}(\maxO_{\bar{\Q}})
	$$
given by $x\mapsto [\zeta_n]$. Given the above, this is the second theorem of our first 
paper~\cite{Borger-deSmit:Integral-models}.

On the other hand, if $\gf=(1)$ but $K$ is arbitrary and $P=\mx_K$, then we have
	$$
	\peru{W}{1}_{\maxO_K,\Pr}(\maxO_{\bar{K}}) = \maxO_H
	$$
where $H$ is the Hilbert class field of $K$. More generally, if $\gf$ is supported at infinity, then
$\peru{W}{\gf}_{\maxO_K,\Pr}(\maxO_{\bar{K}})$ is the ring of integers in the corresponding narrow
Hilbert class field.

\begin{corollary}
\label{cor:periodic-W-class-fields}
	The periodic Witt vector ring is generically a product of ray class fields:
		\begin{equation}
			K\otimes_{\scalA} \peru{W}{\gf}_{\scalA,\Pr}(\intO) = 
			\prod_{\substack{\gd\in\Id_{\Pr}\\\gd\mid\gf_{\fin}}}K(\gf\gd^{-1}).
		\end{equation}
	The ray class field $K(\gf)$ is the image of the projection
		\begin{equation}
			K\otimes_{\scalA} \peru{W}{\gf}_{\scalA,\Pr}(\intO) \longmap \Kb, \quad x\mapsto x_{(1)}.			
		\end{equation}
\end{corollary}

\subsection{} \emph{Remark: Periodic Witt vectors and explicit class field theory.}
\label{subsec:periodic-W-and-explicit-CFT}
It follows from corollary~\ref{cor:periodic-W-class-fields} that any ray class field $K(\gf)$ is generated by
the first coordinate $x_{(1)}$ of the $\gf$-periodic $\intO$-points $x$ on the affine scheme
$W_{\scalA,\Pr}=\Spec(\Lambda_{\scalA,\Pr})$. We emphasize however that $W_{\scalA,\Pr}$ is not of
finite type, and so this does not give an explicit method of producing a polynomial whose roots generate 
$K(\gf)$. In fact, the periodic locus $\peru{W}{\gf}_{\scalA,\Pr}$ is itself not even of finite
type. For instance, if $K=\Q$, $P=\mx_\Q$, and $\gf=\infty$, then $\Lambda^{(\gf)}_{\scalA,\Pr}$ is
isomorphic to the binomial ring, the subring of $\Q[x]$ generated by the binomial coefficients $\binom{x}{n}$,
and this is not finitely generated as a ring. Therefore to find a point of $W^{(\gf)}_{\scalA,\Pr}$, one has to
solve infinitely many simultaneous polynomial equations with coefficients in $\scalA$.

To be sure, $K\otimes_\scalA \Lambda^{(\gf)}_{\scalA,\Pr}$ is finitely generated as a $K$-algebra,
because on $K$-algebras it represents the periodic ghost functor $\ringr\mapsto \ringr^{\DR_\Pr(\gf)}$ and
$\DR_\Pr(\gf)$ is finite. Therefore a periodic Witt vector is determined by finitely many components, namely its
ghost components. However to give a criterion for a periodic ghost vector to be a Witt vector, we need
infinitely many congruences between polynomials in the ghost components. If we add variables to express these 
congruences as equations, we will need infinitely many new variables.

\section{$K=\Q$: the toric line and the Chebyshev line} 

In this section, we consider the case where $\scalA$ is $\Z$ and $\Pr$ is the set of all maximal ideals of $\Z$.
For any integer $n\geq 1$, let us write $\psi_n=\psi_{(n)}$ and $\Lambda=\Lambda_{\Z,\Pr}$.

\subsection{} \emph{The toric line and cyclotomic extensions.}
\label{subsec:toric-line}
Define the toric $\Lambda$-structure on $\Gm=\Spec\Z[x^{\pm 1}]$ to be the one
given by $\psi_{p}(x)=x^p$, for all primes numbers $p$. (We use this name because it is
a particular case of the natural $\Lambda$-structure on any toric variety.)
Observe that it extends uniquely to the affine and projective lines.

For each cycle of the form $(n)\infty$, with $n\geq 1$, the periodic locus
$\per{\Gm}{n\infty}$ is simply $\mu_n=\Spec\Z[x]/(x^n-1)$.
In other words, the containment (\ref{eq:periodic-in-torsion}) is an equality.
Thus we have 
	$$
	\Q\otimes_{\Z}\per{\Gm}{n\infty}=\coprod_{d|n} \Spec \Q(\zeta_d).
	$$
So in this case, the $n\infty$-periodic locus of the $\Lambda$-scheme $\Gm$
does in fact generate the ray class field of conductor $n\infty$.

\subsection{} \emph{The Chebyshev line and real-cyclotomic extensions.}
\label{subsec:chebyshev}
The toric $\Lambda$-ring $\Z[x^{\pm 1}]$ above has an
automorphism $\sigma$ defined by $\sigma(x)=x^{-1}$. The fixed subring is easily seen to be
a sub-$\Lambda$-ring and freely generated as a ring by $y=x+x^{-1}$.
The elements $\psi_{p}(y)\in\Z[y]$ are given by Chebyshev polynomials
$$
\psi_n(y)=2+\prod_{i=0}^{n-1} (y-\zeta_n^i-\zeta_n^{-i}),
$$
where $\zeta_n$ is a primitive
$n$-th root of unity. For example, we have
	$$
	\psi_{2}(y) = y^2-2, \quad\quad \psi_{3}(y)=y^3-3y, \quad\quad \psi_{5}(y)=y^5-5y^3+5y,
	\quad\quad\dots.
	$$
This gives the affine line $Y=\Spec \Z[y]$ a $\Lambda$-scheme structure. We call it the Chebyshev
$\Lambda$-structure. It also extends uniquely to the projective line. (Incidentally, $\Z[y]$ is
isomorphic as a $\Lambda$-ring to the Grothendieck group of the Lie algebra $\mathfrak{sl}_2$. By 
Clauwens's theorem~\cite{Clauwens:Line}, this and the toric $\Lambda$-structure are the only two 
$\Lambda$-structures on the affine line, up to isomorphism.)

Now consider a cycle $\gf$ with trivial infinite part. Write $\gf=(n)$, 
where $n\geq 1$, and write $Y(n)=Y(\gf)$ for the periodic locus.
We have
	$$
	Y[1] = Y(1) = \Spec \Z[y]/(y-2),
	$$
and hence the $n$-torsion locus is
	$$
	Y[n] = \psi_n^{-1} (\Spec \Z[y]/(y-2))
	$$
or more simply, $Y[n]=\psi_n^{-1}(2)$. Observe that $Y[n]$ is not reduced when $n\geq 3$: for
instance if $n$ is odd, we have
	$$
	\psi_n(y)-2 = \prod_{i=0}^{n-1} (y-\zeta_n^i-\zeta_n^{-i}) 
		= (y-2)\prod_{i=1}^{\frac{n-1}{2}} (y-\zeta_n^i-\zeta_n^{-i})^2.
	$$
However the periodic locus is reduced. In fact, we have
	$$
	\Q\otimes_\Z \sheafO(Y(n)) = \prod_{d|n} \Q(\zeta_d+\zeta_d^{-1}),
	$$
by the following more precise integral result:

\begin{proposition}\label{pro:chebyshev-periodic-locus2}
	\begin{enumerate}
		\item The periodic locus $\per{Y}{n}$ is flat and reduced, and the inclusion 
			$$
			i:\per{Y}{n}\to Y[n]_{\red}
			$$ 
			is an isomorphism.
		\item The map $\sheafO(\per{Y}{n})\to \Z[x]/(x^n-1)$ is injective. If $n$ is odd,
			its image is the subring of invariants under the involution $\sigma:x\mapsto x^{-1}$.
			If $n$ is even, its image is the span of $\{1,x+x^{-1},\dots,x^{n/2-1}+x^{1-n/2},2x^{n/2}\}$.
	\end{enumerate}
\end{proposition}
\begin{proof}
(1): Consider the differences of Chebyshev polynomials 
	$$
	P_{a,b}(y) := \psi_a(y)-\psi_b(y) \in \Z[y].
	$$ 
Then the periodic locus $\per{Y}{n}$ is $\Spec\Z[y]/I$, where 
	$$
	I= \big(P_{a,b}(y) \,\mid\, a\equiv \pm b\bmod n,\ a,b\geq 1\big).
	$$
Let $Q(y)$ denote the product of the monic irreducible factors of $(\psi_n(y)-2)$, each taken only with 
multiplicity $1$. If $n$ is odd, we have
	$$
	Q(y) = \prod_{i=0}^{\frac{n-1}{2}}(y-\zeta_n^i-\zeta_n^{-i}).
	$$
and if $n$ is even, we have
	$$
	Q(y) = \prod_{i=0}^{\frac{n}{2}}(y-\zeta_n^i-\zeta_n^{-i}).
	$$
Then we have
	$$
	I\subseteq (Q(y))
	$$
since each $P_{a,b}(y)$ vanishes at $y=\zeta_n^{i}+\zeta_n^{-i}$, for each $i$. 
Therefore to prove (1), it is enough to show this containment is an equality.

Since $\Z[x^{\pm 1}]$ is a faithfully flat extension of $\Z[y]$, it is enough to show this after base change to 
$\Z[x^{\pm 1}]$. In other words, it is enough to show
	$$
	Q(x+x^{-1}) \in \big(P_{a,b}(x+x^{-1}) \,\mid\, a\equiv\pm b\mod n\big).
	$$
We have
	$$
	P_{a,b}(x+x^{-1}) = x^a+x^{-a}-x^b-x^{-b} = x^{-a}(x^{a+b}-1)(x^{a-b}-1).
	$$
Therefore we have
\begin{align*}
	\big(P_{n+1,1}(x+x^{-1})\big) &=\big((x^n-1)(x^{n+2}-1)\big), \\
	\big(P_{n+2,2}(x+x^{-1})\big) &=\big((x^n-1)(x^{n+4}-1)\big)
\end{align*}
and hence
	$$
	\big(P_{n+1,1}(x+x^{-1}),P_{n+2,2}(x+x^{-1})\big) 
	= \big(x^n-1\big)\big((x^{2}-1),(x^{n}-1)\big).
	$$
If $n$ is odd, then we have
	$$
	\big(x^n-1\big)\big((x^2-1),(x^n-1)\big) = \big((x^n-1)(x-1)\big) = \big(Q(x+x^{-1})\big).
	$$
Similarly, if $n$ is even 
	$$
	\big(x^n-1\big)\big((x^2-1),(x^n-1)\big) = \big((x^n-1)(x^2-1)\big) = \big(Q(x+x^{-1})\big).
	$$
Thus in either case, we have 
\begin{align*}
	Q(x+x^{-1}) &\in \big(P_{n+1,1}(x+x^{-1}),P_{n+2,2}(x+x^{-1})\big) \\
	&\subseteq \big(P_{a,b}(y) \,\mid\, a\equiv \pm b\bmod n,\ a,b\geq 1\big).
\end{align*}

(2): Suppose a polynomial $f(y)\in\Z[y]$ maps to zero in $\Z[x]/(x^n-1)$.
Then we have $f(\zeta_n^i+\zeta_n^{-i})=0$ for all $i$.
Therefore we have $(y-\zeta_n^i-\zeta_n^{-i})\mid f(y)$. Thus if $n$ is odd, we have
	$$
	\prod_{i=0}^{\frac{n-1}{2}}(y-\zeta_n^i-\zeta_n^{-1})\mid f(y)
	$$
and if $n$ is even,
	$$
	\prod_{i=0}^{\frac{n}{2}}(y-\zeta_n^i-\zeta_n^{-1})\mid f(y).
	$$
In either case, we have $Q(y)\mid f(y)$. Therefore the map $\Z[y]/(Q(y))\to \Z[x]/(x^n-1)$ is injective.

Let us now consider surjectivity. When $n$ is odd, the set 
	$$
	\{1,x+x^{-1},\dots,x^{\frac{n-1}{2}}+x^{-\frac{n-1}{2}}\}
	$$
is a $\Z$-basis
for the subring $\big(\Z[x]/(x^n-1)\big)^\sigma$ of $\sigma$-invariants, and this is contained in the
image of $\Z[y]$. When $n$ is even, $\{1,x+x^{-1},\dots,x^{n/2-1}+x^{1-n/2},x^{n/2}\}$ is a $\Z$-basis. All
these elements but $x^{n/2}$ lie in the image of $\Z[y]$. However $2x^{n/2}$ does lie in the image. 
\end{proof}

\begin{corollary}
\label{cor:chebyshev-periodic-ray-class-algebra}
	The map 
		$$
		\sheafO(\per{Y}{n}) \longmap \rcl{\Z,\Pr}{n}
		$$ 
	to the ray class algebra given by the point $\zeta_n+\zeta_n^{-1}\in\per{Y}{n}(\maxO_{\bar{\Q}})$ is 
	injective. 
	If $n$ is odd, it is surjective; if $n$ is even, its cokernel is a group of order $2$ generated by
	the class of 
	$[-1]=\langle \dots,(-1)^{d},\dots \rangle\in \prod_{d\mid n}\Q(\zeta_{n/d}+\zeta_{n/d}^{-1})$. 
\end{corollary}
\begin{proof}
By the map $\Z[x]/(x^n-1)\to \maxO_{\bar{\Q}}$ given by $x\mapsto \zeta_n$ results in a diagram
	$$
	\xymatrix{
		\Z[x]/(x^n-1)\ar[r] 
			& \rcl{\Z,\Pr}{n\infty} \\
		\big(\Z[x]/(x^n-1)\big)^{\sigma}\ar[r]\ar@{>->}[u] 
			& \rcl{\Z,\Pr}{n}. \ar@{>->}[u]
	}
	$$
By the second theorem in our first paper~\cite{Borger-deSmit:Integral-models}, the top map is an isomorphism.
Since taking maximal $\Lambda$-orders commutes with taking group invariants,
by part (2) of proposition~\ref{pro:maximality-facts}, the bottom arrow is also an
isomorphism. Observe that when $n$ is even, the image of $x^{n/2}$ is $[-1]$. Now invoke
proposition~\ref{pro:chebyshev-periodic-locus2}. 
\end{proof}

\section{$K$ imaginary quadratic: CM elliptic curves and the Latt\`es scheme}
Let $K$ be an imaginary quadratic field. For convenience, let us fix an embedding $K\subset\C$.
In this section, we show how explicit class field theory over $K$, due to
Kronecker and his followers, can be set naturally in the framework of this paper. This builds on
Gurney's thesis~\cite{Gurney:thesis}. The arguments are similar in spirit to those in the real-cyclotomic
context in the previous section.

Let us write 
$$
\Lambda=\Lambda_{\maxO_K,\mx_K} \quad\text{and}\quad \Lambda_\gp=\Lambda_{\maxO_{K_\gp},\gp},
$$ 
where $\gp\in\Pr$ is any given prime.

\subsection{} \emph{The moduli space of CM elliptic curves.}
Let $R$ be an $\maxO_K$-algebra. Then a \emph{CM elliptic curve} over $R$ is an elliptic curve $E$ over $R$
together with a ring map $\maxO_K\to\End_R(E)$ such that the induced action of $\maxO_K$ on the tangent space
$T_0(E)$ agrees with the $\maxO_K$-algebra structure map $\maxO_K\to R=\End_R(T_0(E))$. Let $\MCM(R)$ denote
the category whose objects are the CM elliptic curves over $R$ and whose morphisms are the
$\maxO_K$-equivariant isomorphisms of elliptic curves. Since all morphisms are isomorphisms, $\MCM(R)$ is a
groupoid by definition. As the $\maxO_K$-algebra $R$ varies, the usual base-change maps make $\MCM$ a fibered
category over the category of affine $\maxO_K$-schemes. Further, this fibered category satisfies effective
descent for the fppf topology because all ingredients in its definition can be expressed in terms which are
fppf-local. In other words, $\MCM$ is a stack. (See~\cite{Laumon-MB:Champs-algebriques} the theory of stacks.)
Let $\sE\to\MCM$ denote the universal object. Then $\sE(R)$ is
the groupoid of pairs $(E,x)$, where $E$ is a CM elliptic curve over $R$, and $x$ is a point of $E(R)$.

Both $\sE$ and $\MCM$ have certain endomorphisms $\psi_\ga$, for any integral ideal $\ga\subseteq \maxO_K$,
defined as follows. 
On $\MCM$, the map $\psi_\ga$ is defined by
	$$
	\psi_\ga: E\mapsto \ga^{-1}\otimes E,
	$$
where the elliptic curve $\ga^{-1}\otimes E$ is the Serre tensor product~\cite{Serre:complex-mult},
defined as a functor by 
	$$
	\big(\ga^{-1}\otimes E\big)(C):=\ga^{-1}\otimes_{\maxO_K}\big(E(C)\big),
	$$
for any algebra $C$ over the base $R$ of $E$. (This is again an elliptic curve: since $\ga^{-1}$ is a direct
summand of $\maxO_K^2$, the functor $\ga^{-1}\otimes E$ is a direct summand of the abelian variety $E^2$.) In particular, if $R=\C$ and the period lattice of $E$ is $L$, so
that $E(\C)=\C/L$, then $\ga^{-1}\otimes E$ is 
the elliptic curve with period lattice $\ga^{-1}\otimes_{\maxO_K}L$, up to canonical isomorphism.

Using the isogeny $E\to \ga^{-1}\otimes E$ defined by $x\mapsto 1\otimes x$, we can then define the 
endomorphism $\psi_\ga:\sE\to\sE$ by
	$$
	\psi_\ga: (E,x) \mapsto (\ga^{-1}\otimes E,1\otimes x).
	$$
Observe that it lies over the endomorphism $\psi_\ga$ of $\MCM$.

The $\psi_\ga$ operators should be thought of as providing a $\Lambda$-structure on $\MCM$ and $\sE$, but
$\MCM$ and $\sE$ are stacks and we will not define what $\Lambda$-structures on stacks are here. Let us just
say that the $\psi_\ga$ operators, whether on $\MCM$ or $\sE$, commute up to a coherent family of canonical
isomorphisms coming from the usual associators $(U\otimes V)\otimes W\to U\otimes (V\otimes W)$ on the category
of $\maxO_K$-modules. Similarly, if $\ga$ is a prime $\gp$, then the reduction of $\psi_\gp$ modulo $\gp$
agrees with the $N(\gp)$-power Frobenius map $F_\gp$, in the sense that for any elliptic curve $E$ over
an $\F_p$-algebra, there is unique isomorphism $\psi_\gp(E)\to F_\gp(E)$ compatible with the canonical maps
from $E$.

\subsection{} \emph{Latt\`es scheme.}
\label{subsec:lattes}
We follow Gurney's thesis~\cite{Gurney:thesis}. Let $\sL$ denote the coarse space underlying $\sE$. It is 
the functor defined, for any $\maxO_K$-algebra $C$, by 
	$$
	\sL(C) = \{\text{local-isomorphism classes of triples $(C',E,x)$}\},
	$$
where $C'$ is an fppf cover of $C$, and $E$ is a CM elliptic curve over $C'$, and $x\in E(C')$ and
where two pairs $(C'_1,E_1,x_1)$ and $(C'_2,E_2,x_2)$ are in the same local-isomorphism class
if $C'_1$ and $C'_2$ have a common 
fppf cover $C''$ such that when pulled back to $C''$, there is an isomorphism between $E_1$ and $E_2$ 
identifying $x_1$ and $x_2$. Observe that if $C=\C$ (or any algebraically closed field), we have
	$$
	\sL(\C) = \{\text{isomorphism classes of pairs $(E,x)$}\},
	$$
where $E$ is a CM elliptic curve over $\C$ and $x\in E(\C)$.
Thus $\sL(\C)$ is the union of $h_K$ copies of $\P^1(\C)$, where $h_K$ is the class number of $K$.

The functor $\sL$ is an $\maxO_K$-scheme (see 4.3.10 of~\cite{Gurney:thesis}), which we
call the \emph{Latt\`es scheme}, and unless we say otherwise we will view it as an $\maxO_K$-scheme.
Note however that the structure map $\sL\to\Spec\maxO_K$
factors naturally through the coarse space underlying $\MCM$, which Gurney shows
(2.6.10 of~\cite{Gurney:thesis}) is $\Spec \maxO_H$, where $H$ is the Hilbert class field
of $K$ in $\C$:
	$$
	\xymatrix{
	\sL \ar@{-->}[rr]\ar[dr] & & \Spec \maxO_H \ar[dl] \\
	& \Spec \maxO_K
	}
	$$
So one can also view $\sL$ as an $\maxO_H$-scheme, and sometimes we will. 

For example, given any CM elliptic curve $E$ over $R$, we have an identification
\begin{equation}
	\label{eq:lattes-quotient}
	\sL \times_{\Spec(\maxO_H)}\Spec(R) = E/\maxO_K^*.
\end{equation}
(This can in fact serve as an alternative definition of $\sL$.) 
The universal map $\sE\to\sL$ is then nothing more than
the classical Weber function expressed in our language. 
Gurney (4.3.16 of~\cite{Gurney:thesis}) proves there exists an isomorphism $\sL\to \bP^1_{\maxO_H}$
of $\maxO_H$-schemes.
(The fact that $\sL$ is a $\bP^1$-bundle
over $\maxO_H$ is expected, but the fact that it is the trivial bundle requires more work.) It
follows that as an $\maxO_K$-scheme, $\sL$ has geometrically disconnected fibers if $h_K>1$.
We also emphasize that there appears to be no canonical isomorphism $\sL\to \bP^1_{\maxO_H}$.
There is a canonical $\maxO_H$-point $\infty\in\sL(\maxO_H)$, which corresponds to pairs of
the form $(E,0)$. The isomorphism $\sL\to\bP^1_{\maxO_H}$ would then typically be chosen to send $\infty$ to
$\infty$, but we cannot make any further restrictions, it seems. So the isomorphism $\sL\to\bP^1_{\maxO_H}$ is
canonically defined only up to the action of the stabilizer group of $\infty$, the semi-direct product
$\maxO_H\rtimes\maxO_H^*$.

The endomorphisms $\psi_\ga$ of $\sE$ induce endomorphisms $\psi_\ga:\sL\to\sL$ (morphisms of 
$\maxO_K$-schemes). Then
$\psi_\gp$ for $\gp$ prime reduces to the $N(\gp)$-power Frobenius map $F_\gp$ modulo $\gp$, and hence the
$\psi_\ga$ define a $\Lambda $-structure on $\sL$. (See 4.3.11 of~\cite{Gurney:thesis}.) Further the map
$\sL\to\Spec\maxO_H$ above is a morphism of $\Lambda$-schemes, where $\maxO_H$ is given its unique
$\Lambda$-ring structure.

When expressed in terms of $\bP^1_{\maxO_H}$, the endomorphisms $\psi_\ga$ are often call the
Latt\`es functions, as for example in Milnor's book~\cite{Milnor:dynamics-book} (page 72),
which explains the name.

\begin{proposition}
	\label{pro:lattes-torsor}
	\begin{enumerate}
		\item $\sL[\gf](\C)$ is the set of isomorphism classes of pairs $(E,x)$, where $E$ is a complex
			CM elliptic curve and $x\in E[\gf](\C)$.
		\item The action of $\Id_{\Pr}$ on $\sL[\gf](\C)$, where $\ga$ acts as $\psi_\ga$,
			factors through an action (necessarily unique)
			of $\DR_{\Pr}(\gf)$ on $\sL[\gf](\C)$. The resulting $\DR_{\Pr}(\gf)$-set is a torsor
			generated by any class of the form
			$(E_0,x_0)$, where $E_0$ is a CM elliptic curve and $x_0$ is a generator of $E_0[\gf](\C)$ as an 
			$\maxO_K$-module.
	\end{enumerate}
\end{proposition}
\begin{proof}
	(1): Any point of $y\in\sL(\C)$ is the isomorphism class of a pair $(E,x)$,
	where $E$ is a complex CM elliptic curve and $x\in E(\C)$. 
	Under $\psi_\gf$, the pair $(E,x)$ maps to $(\gf^{-1}\otimes E, 1\otimes x)$.
	So if $y$ lies in the $\gf$-torsion locus $\sL[\gf](\C)$, the object 
	$(\gf^{-1}\otimes E, 1\otimes x)$ must be isomorphic to one of the form $(E',0)$. 
	Therefore the point $1\otimes x\in (\gf^{-1}\otimes E)(\C)$ is $0$, 
	and hence $x$ is an $\gf$-torsion point of $E$. 

	(2): Let us first show that the orbit of $(E_0,x_0)$ under $\Id_{\Pr}$ is all of 
	$\sL[\gf](\C)$. Let $(E,x)$ be an element of
	$\sL[\gf](\C)$. Then there is an integral ideal $\ga\subseteq\maxO_K$ such that there exists an isomorphism 
	of elliptic curves 
	$f\:\ga^{-1}\otimes E_0 \to E$. We can choose $\ga$ such that it is coprime to $\gf$.
	It follows that the image $1\otimes x_0 \in \ga^{-1}\otimes E_0(\C)$ of $x_0$ is a generator
	of $(\ga^{-1}\otimes E_0)[\gf](\C)$, and hence
	that $f(1\otimes x_0)$ is a generator of $E[\gf](\C)$. In other words, there exists an element 
	$b\in \maxO_K$ such that $bf(1\otimes x_0) = x$. Then we have
		\begin{align*}
		\psi_{b\ga}(E_0,x_0) 
			&= (b^{-1}\ga^{-1}\otimes E_0, 1\otimes x_0) \cong (\ga^{-1}\otimes E_0, b\otimes x_0)\\
			&\cong (E, f(b\otimes x_0)) = (E,b f(1\otimes x_0)) = (E,x).
		\end{align*}
	Therefore $(E_0,x_0)$ generates $\sL[\gf](\C)$ as a $\DR_{\Pr}(\gf)$-set.

	It remains to show 
	$$
	\ga\sim_\gf\gb \Longleftrightarrow  \psi_\ga((E_0,x_0))=\psi_\gb((E_0,x_0)),
	$$
	for any two ideals $\ga,\gb\in\Id_{\Pr}$. 
	
	First consider the direction $\Rightarrow$. The equality $\psi_\ga((E_0,x_0))=\psi_\gb((E_0,x_0))$
	is equivalent to the existence of an isomorphism 
	$(\ga^{-1}\otimes E_0, 1\otimes x_0)\cong(\gb^{-1}\otimes E_0, 1\otimes x_0)$.
	Since $\ga\sim_\gf\gb$, by~\ref{pro:DR-original-def}, there exists
	an element $t\in K$ such that $\ga=\gb t$ and $t-1\in \gf\gb^{-1}$.
	So it is enough to show that the isomorphism
		\begin{equation}
			\label{eq:583}
		\ga^{-1}\otimes E_0 \longlabelmap{t\otimes 1} \gb^{-1}\otimes E_0
		\end{equation}
	given by multiplication by $t\otimes 1$, sends $1\otimes x_0$ on the
	left to $1\otimes x_0$ on the right. In other words, it is enough to show that the two elements 
	$1\otimes x_0, t\otimes x_0\in \gb^{-1}\otimes E_0(\C)$	agree. 	
	
	We give an argument using period lattices.
	Write $E_0(\C)=\C/\gc$, where $\gc$ is a fractional ideal of $K$, and let $y\in\C$ be a coset representative
	of $x_0\in\C/\gc$. Then the morphism (\ref{eq:583}) above is identified with
		$$
		\C/\ga^{-1}\gc \longlabelmap{t} \C/\gb^{-1}\gc.
		$$
	So what we want to show is equivalent to the congruence
		$$
		ty \equiv y \mod \gb^{-1}\gc.
		$$
	But since $x_0$ is an $\gf$-torsion element of $\C/\gc$, we have $y\in\gf^{-1}\gc$, and hence we have
		$$
		(t-1)y \in \gf\gb^{-1}\cdot \gf^{-1}\gc = \gb^{-1}\gc,
		$$
	as desired.

	Now consider the direction $\Leftarrow$. So assume $\psi_\ga(E_0,x_0)\cong \psi_\gb(E_0,x_0)$.
	This means there is an isomorphism 
			$$
			\C/\ga^{-1}\gc \longlabelmap{t} \C/\gb^{-1}\gc,
			$$
	given by multiplication by some element $t\in K^*$ with $(t)=\ga\gb^{-1}$, such that 
	$tx_0=x_0$.
	Therefore the diagram
	$$
	\xymatrix{
		\gf^{-1}\gc/\gc \ar^{\id}[r]\ar[d] & \gf^{-1}\gc/\gc \ar[d] \\
		\C/\ga^{-1}\gc \ar^{t}[r] & \C/\gb^{-1}\gc
	}
	$$
	commutes, since the two ways around the diagram agree on $x_0$, which is a generator of $\gf^{-1}\gc/\gc$.
	Therefore the difference map $t-1\:\gf^{-1}\gc/\gc\to \C/\gb^{-1}\gc$ is zero. This implies
	$(t-1)\gf^{-1}\gc\subseteq \gb^{-1}\gc$ and hence $t-1\in \gf\gb^{-1}$. It follows
	from~\ref{pro:DR-original-def} that $\ga\sim_\gf \gb$.
\end{proof}

\begin{proposition}
	\label{pro:lattes-periodic-torsion}
	Let $\gf$ be an integral ideal in $\maxO_K$. Then we have equalities
		$$
		\perfl{\sL}{\gf} = \perred{\sL}{\gf} =\sL[\gf]_{\red}
		$$
	of closed subschemes of $\sL$.	
\end{proposition}

\begin{proof}
	We have	the following diagram of containments of closed subschemes of $\sL$:
		$$
		\xymatrix{
		\sL(\gf)_{\red} \ar@{>->}^{\alpha}[r]\ar@{>->}^{\beta}[d] &
		\sL(\gf)_{\fl} \ar@{>->}[r]\ar@{>->}[d] & \sL(\gf) \ar@{>->}[d]\\
		\sL[\gf]_{\red} \ar@{>->}[r] & 
		\sL[\gf]_{\fl} \ar@{>->}[r] &
		\sL[\gf]
		}
		$$
	(As an aside, we note that $\sL[\gf]$ is finite flat of degree $N(\gf)$ over $\maxO_H$,
	by~4.3.11 of~\cite{Gurney:thesis},	and hence $\sL[\gf]_{\fl}=\sL[\gf]$.)
	Observe that $\alpha$ is an isomorphism, by~(\ref{pro:periodic-structure2}), and so	
	it is enough to show that $\beta$ is an isomorphism.
	It is enough to show this after base change to $\C$, since $\beta$ is a closed immersion and
	$\sL[\gf]_{\red}$ is flat over $\maxO_K$.
	Further, since the schemes in question are of finite type, it is enough to show that $\beta$ induces a
	surjection on complex points, or in other words that the inclusion 
		$$
		\per{\sL}{\gf}(\C) \hookrightarrow \sL[\gf](\C) 
		$$
	is surjective.
	But this follows from~\ref{pro:lattes-torsor}, which says that $\sL[\gf](\C)$ is $\gf$-periodic.
\end{proof}

\begin{corollary}\label{cor:weber-factorization}
	The map $\sE[\gf]\to\sL$ factors through the closed subscheme $\sL(\gf)_\fl$ of $\sL$.
\end{corollary}
\begin{proof}
	By functoriality, it factors through $\sL[\gf]$. Because $\sE[\gf]$ is reduced, it factors
	further through $\sL[\gf]_\red$. Therefore, by~\ref{pro:lattes-periodic-torsion}, it factors
	through $\sL(\gf)_\fl$.
\end{proof}

\begin{proposition}\label{pro:ell-finite-level}
Let $R$ be a flat $\maxO_K$-algebra over which there is a CM elliptic curve $E$. Fix an ideal 
$\gf\subseteq\maxO_K$, and write $G=\maxO_K^*$. Let
$\sL_R(\gf)_\fl$ denote the base change of $\sL(\gf)_\fl$ from $\maxO_H$ to $R$.
Then we have the following:
\begin{enumerate}
	\item The map $\sheafO(\sL_R(\gf)_\fl)\to \sheafO(E[\gf])^G$ induced by 
		corollary~\ref{cor:weber-factorization} is 	injective. 
	\item Its cokernel is a finitely generated $R/nR$-module, where $n$ is the order of $\maxO_K^*$.
	\item For any sufficiently divisible integer $m$, there exists a (unique) 
		morphism making the following diagram commute:
		$$
		\xymatrix{
		\sheafO(\sL_R(m\gf)_\fl)\ar@{>->}[r] \ar[d]
			 & \sheafO(E[m\gf])^G \ar[d]\ar@{-->}[dl] \\
		\sheafO(\sL_R(\gf)_\fl)\ar@{>->}[r]
			 & \sheafO(E[\gf])^G .
		}
		$$
\end{enumerate} 
\end{proposition}
\begin{proof}
Write $\sL_R=E/G\cong\bP^1_R$. Then $\sL_R(\gf)_\fl$ is a closed subscheme of $\sL_R$, and 
by~\ref{cor:weber-factorization}, we have the following containment of closed subschemes of $E$:
	$$
	E[\gf] \subseteq \sL_R(\gf)_\fl\times_{\sL_R} E.
	$$
Let $I_\gf(R)$ denote the deal sheaf on $\sL_R(\gf)_\fl\times_{\sL_R} E$ corresponding to $E[\gf]$;
so we have the short exact sequence
\begin{equation}
	\label{seq:hiho}
	0 \longmap I_\gf(R) \longmap \sheafO(\sL_R(\gf)_\fl\times_{\sL_R}E) \longmap \sheafO(E[\gf]) \longmap 0.
\end{equation}
Its long exact sequence of group cohomology begins
\begin{equation}
	\label{diag:g-mod-les}
	0 \longmap I_\gf(R)^G \longmap \sheafO(\sL_R(\gf)_\fl\times_{\sL_R}E)^G \longmap \sheafO(E[\gf])^G
	\longmap H^1(G,I_\gf(R)).	
\end{equation}

We would like to simplify this sequence using the fact that the map
\begin{equation}
	\label{map:gme}
	\sheafO(\sL_R(\gf)_\fl)\to \sheafO(\sL_R(\gf)_\fl\times_{\sL_R}E)^G
\end{equation}
is an isomorphism, which we will now show. 
We have the following diagram of $G$-equivariant quasi-coherent sheaves on $\sL_R$ (dropping the usual
direct-image notation for simplicity):
	$$
	\xymatrix{
	0 \ar[r]
		& \sheafO_{\sL_R} \ar[r]\ar@{->>}[d]
		& \sheafO_E \ar[r] \ar@{->>}[d]
		& M \ar[r]\ar@{->>}[d]
		& 0 \\
	0 \ar[r]
		& \sheafO_{\sL_R(\gf)_\fl} \ar[r]
		& \sheafO_{\sL_R(\gf)_\fl\times_{\sL_R}E} \ar[r]
		& \bar{M} \ar[r]
		& 0,
	}
	$$
where the rows are exact,
$M$ and $\bar{M}$ being defined to be the cokernels as shown. We know $\sheafO_{\sL_R}=\sheafO_E^G$, and so
we have $(\Q\otimes_\Z M)^G=0$ and hence $(\Q\otimes_\Z\bar{M})^G=0$. Therefore the map (\ref{map:gme})
becomes an isomorphism after tensoring with $\Q$. To show it is an isomorphism, it is therefore enough to show
it is surjective. Also observe that both sides are torsion free.
So let $b$ be a $G$-invariant element of $\sheafO(\sL_R(\gf)_\fl\times_{\sL_R}E)$.
Then we have $b=a/n$, for some $a\in\sheafO(\sL_R(\gf)_\fl)$ and some integer $n\geq 1$.
Therefore $a$ becomes a multiple of $n$ in $\sheafO(\sL_R(\gf)_\fl\times_{\sL_R}E)$. But because
$\sL_R(\gf)_\fl\times_{\sL_R}E$ is faithfully flat over $\sL_R(\gf)_\fl$, it must be a multiple of $n$ already
in $\sheafO(\sL_R(\gf)_\fl)$. Therefore $b$ is in $\sheafO(\sL_R(\gf)_\fl)$, and
so the map (\ref{map:gme}) is surjective, and hence an isomorphism.

Thus we can rewrite the long exact sequence (\ref{diag:g-mod-les}) as
\begin{equation}
	\label{diag:g-mod-les2}
	0 \longmap I_\gf(R)^G \longmap \sheafO(\sL_R(\gf)_\fl) \longmap \sheafO(E[\gf])^G
	\longmap H^1(G,I_\gf(R)).	
\end{equation}
Using this sequence, we will prove parts (1)--(3).

(1): We will show $I_\gf(R)^G=0$. First, observe that $I_\gf(R)^G$ is a flat-local construction in $R$:
if $R'/R$ is flat, we have $I_\gf(R')=R'\otimes_R I_\gf(R)$ and hence 
	$$
	I_\gf(R')^G = (R'\otimes_R I_\gf(R))^G = R'\otimes_R I_\gf(R)^G
	$$
since taking invariants under the action of a finite group commutes with flat base change.

Second, since $I_\gf(R)$ is a submodule of $\sheafO(\sL_R(\gf)_\fl)$, which by construction
is flat over $O_K$, its invariant subgroup
$I_\gf(R)^G$ is also flat over $\maxO_K$ and hence maps injectively to 
	$$
	\C\otimes_{\maxO_K}I_\gf(R)^G=(\C \otimes_{\maxO_K} R)\otimes_R I_\gf(R)^G=I_\gf(\C\otimes_{\maxO_K}R)^G.
	$$ 
Therefore to show $I_\gf(R)^G=0$, it is enough 
to do it in the case where
$R$ is a $\C$-algebra; and because any $\C$-algebra is flat over $\C$, we can apply the flat-local
property again and conclude that it is enough
to assume $R=\C$, which we will now do. 

Summing up, we have the diagram
$$
	\xymatrix{
	E[\gf] \ar@{>->}[r]
	& \sL_\C(\gf)_\fl\times_{\sL_\C}E \ar@{>->}[r]\ar[d]
	& E \ar[d] \\
	& \sL_\C(\gf)_\fl \ar@{>->}[r]
	& \sL_\C
	}
$$
where $E\to \sL_\C$ is a $G$-Galois cover of complex curves, the horizontal maps are closed immersions,
and $E[\gf]$ and $\sL_\C(\gf)_\fl$ are reduced. Write
$I_\gf(\C)=\bigoplus_Z I_Z$, where $Z$ runs over the $G$-orbits of $E[\gf](\C)$ and the ideal 
$I_Z$ of $\sheafO(\sL_\C(\gf)_\fl\times_{\sL_\C}E)$
is the part of $I_\gf(\C)$ supported at $Z$. It is then enough to show $I_Z^G=0$ for each $Z$.
So consider the filtration of $I_Z$ by its powers:
	$$
	I_Z \supseteq I_Z^2 \supseteq \cdots \supseteq I_Z^e=\{0\},
	$$
where $e$ is the ramification index of $E$ over $\sL$ at $Z$, or equivalently the order of the stabilizer
subgroup $H\subseteq G$ of a point $x\in Z$. 
Then it is enough to show $(I_Z^n/I_Z^{n+1})^G=0$ for $n=1,\dots,e-1$.

This holds vacuously if $e=1$. So assume $e>1$.
For $n\leq e-1$, the $G$-representation $I_Z^n/I_Z^{n+1}$ is the induced representation
$\mathrm{Ind}_{H}^G(U^{\otimes n})$ where and $U=I_x/I_x^2$ is the
cotangent space of $E$ at $x$. Therefore we have
	$$
	(I_Z^n/I_Z^{n+1})^G = (U^{\otimes n})^{H}.
	$$
Observe that as a representation of $H$, the cotangent space $U$ is isomorphic to the restriction to $H$ of the
representation of $G=\maxO_K^*$ on $\C$ given by usual multiplication. Therefore $U^{\otimes n}$ is the
one-dimensional representation on which a generator $\zeta\in H$ acts as multiplication by
$\zeta^n$. But since $n<e$, and since
$e$ is the order of $H$, we have $\zeta^n\neq 1$. Therefore $\zeta$ acts nontrivially on $U^{\otimes n}$ and
hence we have
	$$
	(I_Z^n/I_Z^{n+1})^G = (U^{\otimes n})^H = 0,
	$$
as desired.

(2): By general properties of group cohomology, $H^1(G,I_\gf(R))$ is an $R$-module of exponent dividing $n$,
and hence an $R/nR$-module.
The cokernel in question is a sub-$R$-module of $H^1(G,I_\gf(R))$ and is therefore also an $R/nR$-module.
It remains to show it is finitely generated. Since it is a quotient of $\sheafO(E[\gf])^G$,
it is enough to show that $\sheafO(E[\gf])^G$ is finitely generated, and hence enough to show this
locally on $R$. But locally $E$ descends to some
finitely generated $\maxO_K$-algebra over which $R$ is flat. Since the formation
of $\sheafO(E[\gf])^G$ commutes with flat base change, it is enough to show finite generation in the case
where $R$ is a finitely generated $\maxO_K$-algebra. Here it holds because $E[\gf]$ is 
finite flat and $R$ is noetherian.

(3): Write $I_\gf=I_\gf(R)$. Then by part (1), for any $m\geq 1$,we have a morphism of exact sequences
$$ 
\xymatrix{
0 \ar[r]
	& \sheafO(\sL_R(m\gf)_\fl) \ar[r]  \ar[d]
	& \sheafO(E[m\gf])^G \ar[r] \ar[d]
	& H^1(G,I_{m\gf}) \ar^\alpha[d] \\
0 \ar[r]
	& \sheafO(\sL_R(\gf)_\fl) \ar[r] 
	& \sheafO(E[\gf])^G \ar[r]
	& H^1(G,I_{\gf}). 
}
$$
To prove (3), it is enough to show the map $\alpha$ above is zero for sufficiently divisible $m$. Since $I_\gf$
is torsion free, it is enough by lemma~\ref{lem:H-vanishing} below to show that the map $I_{m\gf}\to I_\gf$ is
zero modulo $n$ for sufficiently divisible $m$. Further, by part (2), it is
enough to show this instead modulo each prime dividing $n$---that is, it is enough to
show that for each prime $p\mid n$, there exists 
an integer $m$ such that the map $I_{m\gf}\to I_\gf$ is zero modulo $p$.

So fix a prime $p\mid n$. Let $E_p$ denote the reduction of $E$ modulo $p$, which is an elliptic curve
over $R/pR$. Then for any $j\geq 0$, we have a morphism of short exact sequences
	$$
	\xymatrix{
	0 \ar[r] 
		& I_{p^j\gf}/pI_{p^j\gf} \ar[r] \ar[d]
		& \sheafO(\sL_R(p^j\gf)_\fl\times_{\sL_R} E_p) \ar[r]\ar[d]
		& \sheafO(E_p[p^f\gf]) \ar[r]\ar[d]
		& 0 \\
	0 \ar[r] 
		& I_{\gf}/pI_{\gf} \ar[r]
		& \sheafO(\sL_R(\gf)_\fl\times_{\sL_R} E_p) \ar[r]
		& \sheafO(E_p[\gf]) \ar[r]
		& 0.
	}
	$$
(Left-exactness is because $E[\gf]$ and $E[p^j\gf]$ are flat.)
Now whether the fibers of $E_p$ are supersingular or ordinary, the closed subschemes
$E_p[p^j\gf]$ contain arbitrary nilpotent thickenings of $E_p[\gf]$, for large enough $j$. Since
$\sL_R(\gf)_{\fl}\times_{\sL}E_p$ is a nilpotent thickening of $E_p[\gf]$, we can take $j$ such that
$E_p[p^j\gf]$ contains $\sL_R(\gf)_{\fl}\times_{\sL}E_p$. Thus the map
	$$
	I_{p^j\gf}/pI_{p^j\gf} \longmap I_{\gf}/pI_{\gf}
	$$
is zero, and hence so is the map $I_{p^j\gf}\to I_\gf/pI_\gf$, as desired.
\end{proof}

\begin{lemma}\label{lem:H-vanishing}
	Let $G$ be a finite group, and let $n$ denote its order. Let $M\to N$ be a morphism 
	of $G$-modules which vanishes modulo $n$, and assume $N$ is $n$-torsion free.
	Then the induced map $H^i(G,M)\to H^i(G,N)$ is zero, for all $i\geq 1$.
\end{lemma}
\begin{proof}
Since our given map $M\to N$ vanishes modulo $n$ and $N$ is $n$-torsion free,
it factors as 
	$$
	\xymatrix{
	M \ar@{-->}[d] \ar[dr] \\
	N \ar^{n}[r] & N.
	}
	$$
Hence the induced map on cohomology factors similarly
	$$
	\xymatrix{
	H^i(G,M) \ar@{-->}[d] \ar[dr] \\
	H^i(G,N) \ar^{n}[r] & H^i(G,N)
	}
	$$
However $H^i(G,N)$ is a $n$-torsion group for $i\geq 1$,
because $G$ has exponent dividing $n$. Therefore the bottom map
is zero and hence so is the diagonal map.
\end{proof}

\begin{corollary}\label{cor:lattes-elliptic-comparison}
	With the notation of proposition~\ref{pro:ell-finite-level}, the map
		$$
		\Big(\sheafO(\sL_R(\gf)_{\fl})\Big)_\gf \longmap \Big(\sheafO(E[\gf])^G\Big)_\gf
		$$
	of pro-rings is an isomorphism.
\end{corollary}

\begin{proposition}\label{pro:local-cm-maximality}
	Let $\gp$ be a prime of $\maxO_K$,	let $R$ be a finite unramified extension of $\maxO_{K_\gp}$, and
	let $E$ be a CM elliptic curve over $R$. Then for any integer $r\geq 0$,
	$\sheafO(E[\gp^r])$ is $\Lambda_{\gp}$-normal over $\maxO_{K_\gp}$.
\end{proposition}
\begin{proof}
	We follow (3.1) of~\cite{Borger-deSmit:Integral-models}, which is the analogous result for $\Gm$.

	For $r=0$, it is clear. So we may assume that $r\geq 1$ and, by induction, that
	$\sheafO(E[\gp^{r-1}])$ is $\Lambda_{\gp}$-normal over $\maxO_{K_\gp}$.
	
	The $\gp^r$-torsion $E[\gp^r]$ can be obtained by gluing the $\gp^{r-1}$-torsion $E[\gp^{r-1}]$ and the
	locus of exact order $\gp^r$, as follows. Let $\ringd_r$ denote $\sheafO(E[\gp^r])$.
	Let $\pi$ be a generator of $\gp\maxO_{K_\gp}$, and consider the Lubin--Tate polynomial
	$h(z)\in R[z]$ for $E$ with respect to the uniformizer $\pi$ and
	some fixed local coordinate. So we have $h(z)=zg(z)$, where $g(z)$ is a monic 
	Eisenstein polynomial with $g(0)=\pi$.
	Then $\ringd_r$ is identified with $R[z]/(h^{\circ r}(z))$. Because 
	$h^{\circ r}(z)=h^{\circ r-1}(z)\cdot g(h^{\circ r-1}(z))$, we have the following diagram
	of quotients of $R[z]$:
		$$
		\xymatrix{
		R[z]/(h^{\circ r}(z)) \ar[d]\ar[r] & R[z]/(h^{\circ r-1}(z)) \ar[d] \\
		R[z]/(g(h^{\circ r-1}(z))) \ar[r] & R[z]/(h^{\circ r-1}(z),g(h^{\circ r-1}(z))).
		}
		$$
	Observe that $g(h^{\circ r-1}(z))$ is a monic Eisenstein polynomial. For degree reasons,
	it does not divide $h^{\circ r-1}(z)$, and so the least common multiple of these two polynomials
	is their product. Therefore we have
		$$
		(h^{\circ r}(z)) = (h^{\circ r-1}(z))\cap (g(h^{\circ r-1}(z))),
		$$
	and hence the diagram above is a pull-back diagram. Further, if we put
	$\ringb_r=R[z]/(g(h^{\circ r-1}(z)))$, then $\ringb_r$ is the ring of integers in
	a totally ramified extension of $K_\gp$.	
	Let $\pi_r\in \ringb_r$ denote the corresponding uniformizer, namely the coset of $z$.
	Thus we have a pull-back diagram
		$$
		\xymatrix{
		\ringd_r \ar@{->>}[d]\ar@{->>}[r] & \ringd_{r-1} \ar@{->>}[d] \\
		\ringb_r \ar@{->>}[r] & \bar{\ringb}_r,
		}
		$$
	where $\bar{\ringb}_r=\ringb_r/(h^{\circ r-1}(\pi_r))$. Further observe that 
	the element $h^{\circ r-1}(\pi_r)\in \ringb_r$ is
	a root of $g(z)$, which is an Eisenstein polynomial of degree $q-1$,
	where we write $q=N(\gp)$. Therefore we have 
	$v_\gp(h^{\circ r-1}(\pi_r))=1/(q-1)$, where $v_\gp$ is the valuation normalized such that $v_\gp(\pi)=1$.
	
	Now suppose $C$ is a sub-$\Lambda_\gp$-ring of $K_\gp\otimes_{\maxO_{K_\gp}} \ringd_r$ which is finite over 
	$\maxO_{K_\gp}$. The maximality statement we wish to prove is that $C$ is contained in 
	$\ringd_r=\ringb_r\times_{\bar{\ringb}_r} \ringd_{r-1}$.
	By induction, $\ringd_{r-1}$ is $\Lambda_\gp$-normal over $\maxO_{K_\gp}$, and hence 
	the image of $C$ in $K_\gp\otimes_{\maxO_{K_\gp}}\ringd_{r-1}$ is contained in $\ringd_{r-1}$.  Similarly, 
	since $\ringb_r$ is a maximal order in the usual sense, the image of $C$ in 
	$K_\gp\otimes_{\maxO_{K_\gp}}\ringb_r$ is contained in $\ringb_r$. Putting the two together, we have the 
	containment $C\subseteq \ringb_r\times \ringd_{r-1}$.
	
	To show $C \subseteq \ringb_r\times_{\bar{\ringb}_r} \ringd_{r-1}$, let us suppose that this does not hold.
	Then there is an element $(b,\overline{f(z)})\in C\subseteq \ringb_r\times \ringd_{r-1}$ such that 
	$b$ and $\overline{f(z)}$ do not become equal in $\bar{\ringb}_r$. In other words, we have 
	$$
	v_\gp(b-f(\pi_r))< 1/(q-1).
	$$
	Further we can choose $(b,\overline{f(z)})\in C$ such that $v_\gp(b-f(\pi_r))$ is as small as possible.
	Write $a=v_\gp(b-f(\pi_r))$.
	Since $C$ is a sub-$\Lambda_\gp$-ring of $K_\gp\otimes_{\maxO_{K_\gp}}\ringd_r$, we know
		$$
		(b,\overline{f(z)})^q-\psi_\gp((b,\overline{f(z)})) \in \pi C
		$$
	The left side can be simplified using the fact $\psi_\gp(z)=h(z)$ on $\ringd_r$:
		$$
		(b,\overline{f(z)})^q-\psi_\gp((b,\overline{f(z)})) 
			= (b^q-f^*(h(\pi_r)),\overline{f(z)}^q-\overline{f^*(h(z))})
		$$
	where $f^*(z)$ the polynomial obtained by applying the Frobenius map to each coefficient of $f(z)$.
	Therefore by the minimality of $a$, we have
	\begin{align*}
		a &\leq v_\gp\Big(\frac{b^q-f^*(h(\pi_r))}{\pi} - \frac{f(z)^q-f^*(h(z))}{\pi}\Big|_{z=\pi_r}\Big) 
		= v_\gp\Big(\frac{b^q-f(\pi_r)^q}{\pi}\Big)  \\
		&= v_\gp\Big(\frac{(b-f(\pi_r))^q}{\pi} + (b-f(\pi_r))\cdot c(b,f(\pi_r))\Big)		
	\end{align*}
	where $c(X,Y)$ denotes the polynomial 
	$$
	\frac{(X^q-Y^q) - (X-Y)^q}{\pi(X-Y)} \in \maxO_{K_\gp}[X,Y].
	$$
	But again by the minimality of $a$, we know $a\leq v_\gp(b-f(\pi_r))$. Therefore we have
		$$
		a\leq v_\gp\Big(\frac{(b-f(\pi_r))^q}{\pi}\Big) = qa-1
		$$
	and hence $a\geq 1/(q-1)$.	This contradicts the assumption $a<1/(q-1)$.
\end{proof}

\begin{lemma}\label{lem:global-cm-maximality}
	Let $\scalA=\maxO_K[1/t]$, for some $t\in\maxO_K$, and let 
	$R$ be a  $\Lambda$-ring, finite \'etale over $A$, over which
	there exists a CM elliptic curve $E$.
	Then for any ideal $\gf\subseteq\maxO_K$, 
	the $\Lambda$-ring $\sheafO(E[\gf])$ is $\Lambda$-normal over $\scalA$. 
\end{lemma}
\begin{proof}
	Let $S$ denote the maximal $\Lambda$-order over $A$ 
	in $K\otimes_{\maxO_K}\sheafO(E[\gf])$. 
	Then we have an inclusion $\sheafO(E[\gf])\subseteq S$ of finite flat $\scalA$-algebras. 
	Therefore it is an equality if 
		$\maxO_{K_\gp}\otimes_\scalA \sheafO(E[\gf])=\maxO_{K_\gp}\otimes_\scalA S$
	for all primes $\gp\nmid t$. Thus it is enough to show that $\maxO_{K_\gp}\otimes_\scalA \sheafO(E[\gf])$ is
	$\Lambda_\gp$-normal over $\maxO_{K_\gp}$ for all $\gp\nmid t$.
	
	Write $\gf=\gp^r\ggg$, where $\gp\nmid\ggg$. Then putting $R_\gp=\maxO_{K_\gp}\otimes_\scalA R$, we have
	\begin{align*}
		\maxO_{K_\gp}\otimes_\scalA \sheafO(E[\gf]) 
			&= \maxO_{K_\gp}\otimes_\scalA (\sheafO(E[\gp^r])\otimes_R \sheafO(E[\ggg])) \\
		 	&= (\maxO_{K_\gp}\otimes_\scalA \sheafO(E[\gp^r]))\otimes_{R_\gp} 
				(\maxO_{K_\gp}\otimes_\scalA \sheafO(E[\ggg])) \\
			&= \sheafO(E_{R_\gp}[\gp^r]) \otimes_{R_\gp} \sheafO(E_{R_\gp}[\ggg]) \\
			&= \prod_{R'} \sheafO(E_{R'}[\gp^r]), 
	\end{align*}
	where $R'$ runs over all irreducible direct factors of the finite \'etale $\maxO_{K_\gp}$-algebra 
	$\sheafO(E_{R_\gp}[\ggg])$.
	Therefore it is enough to show that each $\sheafO(E_{R'}[\gp^r])$ 
	is $\Lambda_\gp$-normal over $\maxO_{K_\gp}$. 
	But this follows from~\ref{pro:local-cm-maximality}.
\end{proof}

\begin{proposition}\label{pro:cm-maximality}
	Let $\scalA=\maxO_K[1/t]$, for some $t\in\maxO_K$. Put $R=\maxO_H[1/t]$ and assume
	there exists a CM elliptic curve $E$ over $R$. Write $G=\maxO_K^*$.
	Then the maximal $\Lambda$-order in 
	$K\otimes_{\scalA}\sheafO(\sL_R(\gf)_\fl)$ over $\scalA$ is $\sheafO(E[\gf])^G$.
\end{proposition}
\begin{proof}
	By lemma~\ref{lem:global-cm-maximality} and part (2) of proposition~\ref{pro:maximality-facts},
	the $G$-invariant subring $\sheafO(E[\gf])^{G}$ is 
	$\Lambda$-normal over $\scalA$. Since
	$K\otimes_{\scalA}\sheafO(\sL_R(\gf)_\fl)= K\otimes_{A}\sheafO(E[\gf])^G$, by part (2) of 
	proposition~\ref{pro:ell-finite-level},
	we can conclude that $\sheafO(E[\gf])^{G}$ is the maximal $\Lambda$-order over $\scalA$ in
	$K\otimes_{\scalA}\sheafO(\sL_R(\gf)_\fl)$.
\end{proof}

\begin{thm}\label{thm:cm-pro-ring}
	Let $\sheafO(\sL(\gf)_{\fl})^{\sim}$ denote the maximal $\Lambda$-order over $\maxO_K$
	in $K\otimes_{\maxO_K}\sheafO(\sL(\gf)_{\fl})$. Then 
	\begin{enumerate}
		\item $\sheafO(\sL(\gf)_{\fl})^{\sim}$ is isomorphic as a $\Lambda$-ring to 
			the ray class algebra $\rcl{\maxO_K,\Pr}{\gf}$.
		\item The map 
			$$
			\big(\sheafO(\sL(\gf)_{\fl})\big)_\gf \longlabelmap{}
				\big(\sheafO(\sL(\gf)_{\fl})^{\sim}\big)_\gf
			$$
			is an isomorphism of pro-rings.
		\item If a prime number divides the index $[\sheafO(\sL(\gf)_{\fl})^{\sim}:\sheafO(\sL(\gf)_{\fl})]$
		then it divides the order of the unit group $\maxO_K^*$.
	\end{enumerate} 
\end{thm}
\begin{proof}
	(1): Since $\sheafO(\sL(\gf)_{\fl})^{\sim}$ is the maximal $\Lambda$-order over $\maxO_K$ in 
	$K\otimes_{\maxO_K}\sheafO(\sL(\gf)_{\fl})$, and since $\rcl{\maxO_K,\Pr}{\gf}$ is that in
	$K\otimes_{\maxO_K}\rcl{\maxO_K,\Pr}{\gf}$, it is enough to show there is an isomorphism
		$$
		\sL(\gf)(\bar{\Q})\cong \DR_{\Pr}(\gf)
		$$
	of $\Id_{\Pr}$-sets, but this is follows immediately from part (2) of proposition~\ref{pro:lattes-torsor}.
	
	(2):
	Since the rings 
	$$
	\scalA=\maxO_K[\gq^{-1}]
	$$
	for any two primes $\gq\subset \maxO_K$ form a cover of $\maxO_K$, it is enough
	to prove the map of pro-rings is an isomorphism after base change from $\maxO_K$ to any two such $\scalA$.
	By Shimura's theorem (\cite{Gurney:thesis}, prop.\ 4.2.2), there are at least two primes 
	$\gq\subset\maxO_K$ such that there exists a CM elliptic curve over $R=\maxO_H\otimes_{\maxO_K}A$.
	(In fact, he proves there are infinitely many.)
	So fixing $\scalA$ and a CM elliptic curve $E$ over $\maxO_H\otimes_{\maxO_K}A$, it is enough to show that
		\begin{equation}
			\label{eq:sdfkl}
			\Big(\scalA\otimes_{\maxO_K}\sheafO(\sL(\gf)_{\fl})\Big)_\gf \longlabelmap{}
				\Big(\scalA\otimes_{\maxO_K}\sheafO(\sL(\gf)_{\fl})^{\sim}\Big)_\gf
		\end{equation}
	is an isomorphism of pro-rings.

	But we also have 
		$$
		\scalA\otimes_{\maxO_K}\sheafO(\sL(\gf)_\fl) = R\otimes_{\maxO_H} \sheafO(\sL(\gf)_\fl)
		= \sheafO(\sL_R(\gf)_{\fl})
		$$
	and hence by part (1) of proposition~\ref{pro:maximality-facts}
		$$
		\sheafO(\sL_R(\gf)_{\fl})^{\sim} =
		\scalA\otimes_{\maxO_K}\sheafO(\sL(\gf)_{\fl})^{\sim}.
		$$
	Therefore (\ref{eq:sdfkl}) can be rewritten as 
		\begin{equation}
			\label{eq:sdfkl2}
			\Big(\sheafO(\sL_R(\gf)_{\fl})\Big)_\gf \longlabelmap{}
				\Big(\sheafO(\sL_R(\gf)_{\fl})^{\sim}\Big)_\gf.
		\end{equation}
	We know further that the map 
		$$
		\sheafO(\sL_R(\gf)_{\fl})^{\sim} \longmap \sheafO(E[\gf])^G
		$$
	is an isomorphism by~\ref{pro:cm-maximality}.
	So the map (\ref{eq:sdfkl2}) can in turn be rewritten as
		\begin{equation}
			\label{eq:sdfkl3}
			\Big(\sheafO(\sL_R(\gf)_{\fl})\Big)_\gf \longlabelmap{}
				\Big(\sheafO(E[\gf])^G\Big)_\gf,
		\end{equation}
	and this is an isomorphism of pro-rings by~\ref{cor:lattes-elliptic-comparison}.
	
(3): It is enough to show these properties locally on $\maxO_K$. So as above, it is enough to show
them after base change to $A=\maxO_K[\gq^{-1}]$, where $\gq$ is a prime of $\maxO_K$ such that 
there is a CM elliptic curve $E$ over $R=\maxO_H\otimes_{\maxO_K}A$.

Then the map $A\otimes_{\maxO_K}\sheafO(\sL_R(\gf)_\fl) \to A\otimes_{\maxO_K}\sheafO(\sL_R(\gf)_\fl)^{\sim}$ is 
identified with 
	$$
	\sheafO(\sL_R(\gf)_\fl) \longmap \sheafO(E[\gf])^G.
	$$
By part (2) of proposition~\ref{pro:ell-finite-level}, 
the image of this map is of finite index divisible only by the primes dividing
the order of $\maxO_K^*$.
\end{proof}

\section{Further questions}
Is it possible to use $\Lambda$-schemes of finite type to generate other large abelian extensions, beyond the
Kroneckerian explicit class field theories?  We will formulate a range of such questions in this section,
some of which it is reasonable to hope have a positive answer and some which are more ambitious.

\subsection{} \emph{$\Lambda$-geometric field extensions.}
So far, we have mostly been interested in what in the introduction we called \emph{$\Lambda$-refinements} of
explicit class field theories---that is, in generating ray class algebras instead of abelian field extensions.
But hungry for positive answers, we will give weaker, field-theoretic formulations here.

Let $K$ be a number field, and write 
	$$
	\Lambda=\Lambda_{\maxO_K,\mx_K}\quad \text{and}\quad \rcl{}{\gf}=\rcl{\maxO_K,\mx_K}{\gf}.
	$$
For any separated $\Lambda$-scheme $X$ of finite type over $\maxO_K$, consider
the extension of $K$ obtained by adjoining the coordinates of the $\gf$-periodic points, for all cycles $\gf$:
	$$
	K(X) := \bigcup_\gf K(X(\gf)(\bar{K})).
	$$
For instance, we have the following: 
	\begin{enumerate}
		\item If $K=\Q$ and $X$ is $\Gm$ with the toric $\Lambda$-structure, then
			Then $K(X)$ is the maximal abelian extension $\bigcup_n \Q(\zeta_n)$.
		\item If $K=\Q$ and $X=\A^1$ with the Chebyshev $\Lambda$-structure, then
			$K(X)$ is the maximal totally real abelian extension $\bigcup_n \Q(\zeta_n+\zeta_n^{-1})$.
		\item $K$ is imaginary quadratic, $X$ is $\P^1_{\maxO_H}$ with the Latt\`es $\Lambda$-structure.
			Then $K(X)$ is the maximal abelian extension of $K$.
		\item If $K$ is general and $X=\Spec \rcl{}{\gf}$, then $K(X)$ is the ray class field $K(\gf)$.
		\item $K(X_1\smcoprod X_2)$ is the compositum $K(X_1)K(X_2)$
	\end{enumerate}

Let us say that an abelian extension $L/K$ is
\emph{$\Lambda$-geometric} if there exists an $X$ as above such that $L\subseteq K(X)$. First observe that any
finite extension of a $\Lambda$-geometric extension is $\Lambda$-geometric, by (4) and (5) above. Therefore
there is no maximal $\Lambda$-geometric extension unless, as in the examples above, the maximal abelian
extension itself is $\Lambda$-geometric.

\begin{enumerate}
	\item[(Q4)] Are there number fields other than $\Q$ and imaginary quadratic fields for which the
	maximal abelian extension is $\Lambda$-geometric?
\end{enumerate}

It is natural to consider Shimura's generalization of Kronecker's theory to CM fields and abelian varieties. We
expect that it can be realized in our framework, or at least that some version of it. But note that,
assuming $[K:\Q]>2$, the maximal abelian extension generated by Shimura's method is an infinite subextension of
the maximal abelian extension itself---the relative Galois group is an infinite group of exponent $2$.
(See~\cite{Shimura:class-fields}\cite{Ovseevich:cm-type}\cite{Wei:cm-motives}.)

\begin{enumerate}
	\item[(Q5)] Let $K$ be a CM field of degree greater than $2$. Is Shimura's extension $\Lambda$-geometric?
	If so, is there an infinite extension of it which is $\Lambda$-geometric?
\end{enumerate}

\subsection{} \emph{Production of $\Lambda$-schemes from ray class algebras.}
It appears difficult to find $\Lambda$-schemes of finite type which generate large infinite abelian extensions.
Every example we know ultimately comes from varieties with complex multiplication.
Here we will consider an alternative---the possibility of manufacturing $\Lambda$-schemes of finite type
by interpolating the ray class schemes $\Spec(\rcl{}{\gf})$, in the way that $\Gm$ can be viewed
as interpolating the $\mu_n$ as $n$ varies. This raises some questions which have the flavor of algebraic number
theory more than the geometric questions above, and hence have a special appeal.

Let $\gr$ be a product of real places of $K$, and let $X$ a reduced
$\Lambda$-scheme of finite type over $\maxO_K$. Assume further that the
union
	$$
	\bigcup_{\gf\in\cyc}X(\gf)
	$$
of the closed subschemes $X(\gf)$ is Zariski dense in $X$. If it is not, replace $X$ with the closure.
Then all the information needed to construct $X$ is in principle available inside the function algebra of
	$$
	\colim_{\gf\in\cyc} \per{X}{\gf}.
	$$
To be sure, this ind-scheme and $X$ are quite far apart, much as an abelian variety and its $p$-divisible groups 
are. But pressing on, if $X$ satisfies the property in (Q2) in the introduction, then this colimit would be 
isomorphic to
	$$
	\colim_{\gf\in\cyc}\big(\Spec\rcl{}{\gf}\big),
	$$
and so all the information need to construct $X$ is in principle available in projective limit
	$$
	\perv{\gr}{}
	=\lim_{\gf\in\cyc}\rcl{}{\gf},
	$$
which can be viewed as a construction purely in the world of algebraic number theory in that 
it depends only on $K$ and $\gr$ and not on any variety $X$.

For example in the cyclotomic context, where $K=\Q$ and $\gr=\infty$, the injective map
	$$
	\Z[x^{\pm 1}] \longmap \lim_n \Z[x]/(x^n-1) = \perv{\infty}{}
	$$
realizes the function algebra of $\Gm$ as a dense finitely generated sub-$\Lambda$-ring of 
$\perv{\infty}{}$. 

We can ask whether similar subrings exist for general $K$:
\begin{enumerate}
	\item[(Q6)] Does $\perv{\gr}{}$ have a dense sub-$\maxO_K$-algebra which
		is finitely generated as an $\maxO_K$-algebra? Does it have a dense sub-$\Lambda$-ring which
		is finitely generated as an $\maxO_K$-algebra? 
\end{enumerate}
It might be possible to cook up such a subring purely algebraically, instead of going through geometry.

\subsection{}\emph{The cotangent space.}
One first step in finding such a subring might be to guess its dimension by looking at the cotangent space
of the ray class algebras at a point modulo $\gp$. For example, the cotangent space of $\mu_{p^n}$
modulo $p$ at the origin is $1$-dimensional, at least if $n\geq 1$.

For any ideal $\ga\in\Id_{\mx_K}$ (and $\gr$ still a product of real places),
let $I_\ga$ denote the kernel of the morphism
	\begin{equation}
	\label{eq:cotan}
	\rcl{}{\gr\ga}\longlabelmap{\psi_\ga} \rcl{}{\gr}=\maxO_{K(\gr)}.
	\end{equation}
So we have an exact sequence
	$$
	0\longmap I_\ga \longmap \rcl{}{\gr\ga} \longlabelmap{\psi_\ga} 
		\maxO_{K(\gr)} \longmap 0.
	$$
Therefore $I_\ga/I_\ga^2$ is naturally an $\maxO_{K(\gr)}$-module. It is finitely generated since
$\rcl{}{\gr\ga}$ is noetherian, being finite over $\Z$.

\begin{enumerate}
	\item[(Q7)] Given a maximal ideal $\gq\subset \maxO_{K(\gr)}$ with residue field $k$, is the dimension
		$$
		\dim_k(k\otimes_{\maxO_{K(\gr)}}I_\ga/I_\ga^2)
		$$
		constant for large $\ga$? 
\end{enumerate}
If so, can it be expressed in terms of the classical algebraic number theoretic invariants of $K$?
One might hope it is the number of places of $K$ at infinity.

\bibliography{references}
\bibliographystyle{plain}

\end{document}